\documentclass[12pt,a4paper,twoside,reqno]{amsart}

\title[Pimsner algebras and Gysin sequences]{Pimsner algebras and Gysin sequences \\ ~\\  from principal circle actions}
\date{v1 September 2014; v2 February 2015}

\author{Francesca Arici, Jens Kaad, Giovanni Landi}
\address{International School of Advanced Studies (SISSA), Via Bonomea 265, 34136 Trieste, Italy} 
\address{International School of Advanced Studies (SISSA), Via Bonomea 265, 34136 Trieste, Italy} 
\address{
Matematica, Universit\`{a} di Trieste, Via A.~Valerio 12/1, 34127 Trieste, Italy and INFN, Sezione di Trieste, Trieste, Italy}
\email{farici@sissa.it, jenskaad@hotmail.com, landi@units.it}
\thanks{{\it Thanks.} GL was partially supported by the Italian Project ``Prin 2010-11 -- Operator Algebras, Noncommutative Geometry and Applications''.
}

\usepackage[all,cmtip]{xy}
\usepackage{slashed}
\usepackage{amsmath,amscd}
\usepackage{amssymb,amsfonts,mathrsfs}
\usepackage[driver=pdftex,margin=3cm,heightrounded=true,centering]{geometry}
\usepackage[colorlinks=true,linkcolor=black,citecolor=black]{hyperref}

\usepackage{graphicx}

\theoremstyle{plain}

\newtheorem{theorem}{Theorem}[section]
\newtheorem{prop}[theorem]{Proposition}
\newtheorem{lemma}[theorem]{Lemma}
\newtheorem{cor}[theorem]{Corollary}

\newtheorem{assu}[theorem]{Assumption}

\theoremstyle{definition}
\newtheorem{dfn}[theorem]{Definition}

\newtheorem{remark}[theorem]{Remark}

\theoremstyle{remark} 

\theoremstyle{plain}

\numberwithin{equation}{section}

\newcommand{\alpheqn}[1][\relax]{
     \refstepcounter{equation}
     \if#1\relax \relax
       \else \label{#1}
     \fi  
     \setcounter{saveeqn}{\value{equation}}%
    \setcounter{equation}{0}%
    \renewcommand{\theequation}{\thealphequation}}
\newcommand{\reseteqn}{\setcounter{equation}{\value{saveeqn}}%
     \renewcommand{\theequation}{\thearabicequation}}

\providecommand{\mathscr}{\mathcal} 
\IfFileExists{mathrsfs.sty}{
\usepackage{mathrsfs}}         
   {
     \IfFileExists{eucal.sty}{
        \usepackage[mathscr]{eucal}} 
   {
   }
}

\newcommand{\cd}{\cdot}
\newcommand{\clc}{\cdot\ldots\cdot}
\newcommand{\ot}{\otimes}
\newcommand{\hot}{\widehat \otimes}

\newcommand{\op}{\oplus}

\newcommand{\olo}{\otimes\ldots\otimes}
\newcommand{\plp}{+ \ldots +}

\newcommand{\ci}{\circ}

\newcommand{\ti}{\times}
\newcommand{\nn}{\mathbb{N}}
\newcommand{\zz}{\mathbb{Z}}

\newcommand{\cc}{\mathbb{C}}

\newcommand{\ee}{\mathbb{E}}
\newcommand{\al}{\alpha}
\newcommand{\be}{\beta}
\newcommand{\ga}{\gamma}

\newcommand{\de}{\delta}
\newcommand{\De}{\Delta}
\newcommand{\ep}{\varepsilon}

\newcommand{\io}{\iota}

\newcommand{\la}{\lambda}

\newcommand{\si}{\sigma}

\newcommand{\te}{\theta}

\newcommand{\pa}{\partial}

\newcommand{\C}[1]{\mathcal{#1}}

\newcommand{\T}[1]{\textup{#1}}

\newcommand{\fork}[2]{\left\{ \begin{array}{#1} #2 \end{array} \right.} 

\newcommand{\ma}[2]{\left(\begin{array}{#1} #2 \end{array} \right)}

\newcommand{\lrar}{\Leftrightarrow}

\newcommand{\rar}{\Rightarrow}

\newcommand{\su}{\subseteq}

\newcommand{\q}{\qquad}

\newcommand{\wit}{\widetilde}

\newcommand{\inn}[1]{\langle #1 \rangle}

\newcommand{\dd}{\mathrm{d}}
\newcommand{\ii}{\mathrm{i}}
\newcommand{\M}{\mbox}

\newcommand\sA{\mathscr{A}}
\newcommand\sB{\mathscr{B}}
\newcommand\sF{\mathscr{F}}
\newcommand\sK{\mathscr{K}}
\newcommand\sL{\mathscr{L}}

\begin{document}


\begin{abstract}
A self Morita equivalence over an algebra $B$, given by a $B$-bimodule $E$, is thought of as a line bundle over $B$.
The corresponding Pimsner algebra $\C O_E$ is then the total space algebra of a noncommutative principal circle bundle over $B$. A natural Gysin-like sequence relates the $KK$-theories of $\C O_E$ and of $B$. Interesting examples come from $\C O_E$ a quantum lens space over 
$B$ a quantum weighted projective line (with arbitrary weights). The $KK$-theory of these spaces is explicitly computed and natural generators are exhibited.
\end{abstract}
\maketitle

\tableofcontents
\parskip 1ex
\linespread{1.1}

\section{Introduction}

In the present paper we put in close relation two notions that seem to have touched each other only occasionally in the recent literature.
These are the notion of a Pimsner (or Cuntz-Krieger-Pimsner) algebra on one hand and that of a noncommutative (in general) principal circle bundle on the other.

At the $C^*$-algebraic level one needs a self Morita equivalence of a $C^*$-algebra $B$, thus we look at a full Hilbert $C^*$-module $E$ over $B$ together with an isomorphism of $B$ with the compacts on $E$. Through a natural universal construction this data gives rise to a $C^*$-algebra, the \emph{Pimsner algebra} $\C O_E$ generated by $E$. In the case where both $E$ and its Hilbert $C^*$-module dual $E^*$ are finitely generated projective over $B$ one obtains that the $*$-subalgebra generated by the elements of $E$ and $B$ becomes the total space of a noncommutative principal circle bundle with base space $B$.

At the purely algebraic level we start from a $\zz$-graded $*$-algebra $\sA$ which forms the total space of a \emph{quantum principal circle bundle} with base space the $*$-subalgebra of invariant elements $\sA_{(0)}$ and with a coaction of the Hopf algebra $\C O(U(1))$ coming from the $\zz$-grading. Provided that $\sA$ comes equipped with a $C^*$-norm, which is compatible with the circle action likewise defined by the $\zz$-grading, we show that the closure of $\sA$ has the structure of a Pimsner algebra. Indeed, the first spectral subspace $\sA_{(1)}$ is then finitely generated and projective over the algebra $\sA_{(0)}$. The closure $E$ of $\sA_{(1)}$ will become a Hilbert $C^*$-module over $B$, the closure of $\sA_{(0)}$, and the couple $(E,B)$ will lend itself to a Pimsner algebra construction.

The commutative version of this part of our program was spelled out in \cite[Prop.~5.8]{GG13}. This amounts to showing that the continuous functions on the total space of a (compact) principal circle bundle can be described as a Pimsner algebra generated by a classical line bundle over the compact base space.

With a Pimsner algebra there come two natural six term exact sequences in $KK$-theory, which relate the $KK$-theories of the Pimsner algebra $\C O_E$ with that of the $C^*$-algebra of (the base space) scalars $B$. The corresponding sequences in $K$-theory are noncommutative analogues of the Gysin sequence which in the commutative case relates the $K$-theories of the total space and of the base space. The classical cup product with the Euler-class is in the noncommutative setting replaced by a Kasparov product with the identity minus the generating Hilbert $C^*$-module $E$.
Predecessors of these six term exact sequences are the Pimsner-Voiculescu six term exact sequences of \cite{PiVo:EKE} for crossed products by the integers.

Interesting examples are quantum lens spaces over quantum weighted projective lines. The latter spaces $W_q(k,l)$ are defined as fixed points of weighted circle actions on the quantum $3$-sphere $S_q^3$. On the other hand, quantum lens spaces $L_q(dlk;k,l)$ are fixed points for the action of a finite cyclic group on $S_q^3$. For general $(k,l)$ coprime positive integers and any positive integer $d$, the coordinate algebra of the lens space is a quantum principal circle bundle over the corresponding coordinate algebra for the quantum weighted projective space, thus generalizing the cases studied in \cite{BrFa:QT}.

At the $C^*$-algebra level the lens spaces are given as Pimsner algebras over the $C^*$-algebra 
of the continuous functions over the weighted projective spaces (see \S\ref{se:qls}). 
Using the associated exact sequences coming from the construction of \cite{Pim:CCC}, we explicitly 
compute in \S\ref{s:kktlen} the $KK$-theory of these spaces for general weights. 
A central character in this computation is played by an integer matrix whose entries are index pairings. These are in turn computed by pairing the corresponding Chern-Connes characters in cyclic theory. The computation of the $KK$-theory of our class of $q$-deformed lens spaces is,  
to the best of our knowledge, a novel one.
Also, it is worth emphasizing that the quantum lens spaces and weighted projective spaces are in general not $KK$-equivalent to their commutative counterparts.

Pimsner algebras were introduced in \cite{Pim:CCC}. This notion gives a unifying framework for a range of important $C^*$-algebras including crossed products by the integers, Cuntz-Krieger algebras \cite{CuKr:TMC,Cun:SGI}, and $C^*$-algebras associated to partial automorphisms \cite{Exe:CPP}. 
Generalized crossed products, a notion which is somewhat easier to handle, were independently invented in \cite{AEE:MEC}. 
More recently, Katsura has constructed Pimsner algebras for general $C^*$-correspondences \cite{Ka04}. 
In the present paper we work in a simplified setting 
(see Assumption~\ref{a:hilmod} below) which is close to the one of \cite{AEE:MEC}.

\subsubsection*{Acknowledgments} 
We are very grateful to Georges Skandalis for many suggestions and to Ralf Meyer for useful discussions. We thank 
Tomasz Brzezi{\'n}ski for making us aware of the reference \cite{NaVO82}.
This paper was finished at the Hausdorff Research Institute for Mathematics in Bonn during the 2014 Trimester Program 
``Non-commutative Geometry and its Applications". We thank the organizers of the Program for the kind invitation and all people 
at HIM for the nice hospitality. 

\section{Pimsner algebras}

We start by reviewing the construction of Pimsner algebras associated to Hilbert $C^*$-modules as given in \cite{Pim:CCC}.
Rather than the full fledged generality we aim at a somewhat simplified version adapted to the context of the present paper, and motivated by  our geometric intuition coming from principal circle bundles.

Our reference for the theory of Hilbert $C^*$-modules is \cite{La95}.
Throughout this section $E$ will be a countably generated (right) Hilbert $C^*$-module over a separable $C^*$-algebra $B$,
with $B$-valued (and right $B$-linear) inner product denoted $\inn{\cdot,\cdot}_B$; or simply $\inn{\cdot,\cdot}$ to lighten notations. Also, $E$ is taken to be full, that is the ideal $\inn{E,E} := \T{span}_{\cc}\big\{ \inn{\xi,\eta} \, | \, \xi,\eta \in E \big\}$ is dense in $B$.  

Given two Hilbert $C^*$-modules $E$ and $F$ over the same algebra $B$, we denote by $\sL(E, F)$ the space of bounded \emph{adjointable}
homomorphisms $T: E \to F$. For each of these there exists a homomorphism $T^*: F \to E$ (the adjoint) with the property that $\inn{T^* \xi,\eta} = \inn{\xi, T \eta}$ for any $\xi \in F$ and $\eta \in E$. 
Given any pair $\xi \in F, \eta \in E$, an adjointable operator $\theta_{\xi, \eta} : E \rightarrow F$ is defined by
$$
\theta_{\xi, \eta} (\zeta) =  \xi  \inn{\eta, \zeta} \, , \quad \forall \, \zeta \in E \, .
$$
The closed linear subspace of $\sL(E,F)$ spanned by elements of the form $\theta_{\xi, \eta}$ as above is denoted $\sK (E,F)$, the space of compact homomorphisms. When $E=F$, it results that $\sL(E) := \sL(E,E)$ is a $C^*$-algebra with $\sK(E) := \sK(E,E) \subseteq \sL(E)$ the
(sub) $C^*$-algebra of compact endomorphisms of $E$. 

\subsection{The algebras and their universal properties}
On top of the above basic conditions, the following will remain in effect as well:
 
\begin{assu}\label{a:hilmod} 
There is a $*$-homomorphism $\phi : B \to \sL(E)$ which induces an isomorphism $\phi : B \to \sK(E)$. 
\end{assu}
Next, let $E^*$ be the dual of $E$ (when viewed as a Hilbert $C^*$-module): 
\[
E^* := \big\{ \phi \in \T{Hom}_B(E,B) \, | \, \exists \, \xi \in E \T{ with } \phi(\eta) = \inn{\xi,\eta} \, \, \forall \eta \in E \big\}  \, .
\] 
Thus, with $\xi \in E$, if $\lambda_\xi: E \rightarrow B$ is the operator defined by $\lambda_\xi (\eta) = \inn{\xi, \eta} $, for all $\eta \in E$, 
every element of $E^*$ is of the form $\la_\xi$ for some $\xi\in E$.
By its definition, $E^* := \sK(E,B)$.
The dual $E^*$ can be given the structure of a (right) Hilbert $C^*$-module over $B$. Firstly, the right action of $B$ on $E^*$ is given by
\[
\la_\xi \, b := \la_\xi \ci \phi(b)  \, .
\]
Then, with operator $\te_{\xi,\eta} \in \sK(E)$ for $\xi,\eta \in E$, the inner product on $E^*$ is given by
\[
\inn{\la_\xi, \la_\eta} := \phi^{-1}(\te_{\xi,\eta})  \, ,
\]
and $E^*$ is full as well. With the $*$-homomorphism $\phi^* : B \to \sL(E^*)$ defined by $\phi^*(b)(\la_\xi) := \la_{\xi \cd b^*}$, the pair $(\phi^*,E^*)$ satisfies the conditions in Assumption \ref{a:hilmod}.

We need the interior tensor product  $E\hot_\phi E$ of $E$ with itself over $B$. As a first step, one constructs the quotient of the vector space tensor product 
$E\otimes_{\rm alg}E$ by the ideal generated by elements of the form
\begin{equation}\label{ns}
\xi b \otimes \eta - \xi \otimes \phi(b) \eta \, , \qquad \textup{for}  \quad \xi, \eta \in E \, , \quad b \in B \, . 
\end{equation}
There is a natural structure of right module over $B$ with the action given by 
\[
(\xi \otimes \eta) b = \xi \otimes (\eta b) \, , \qquad \textup{for}  \quad \xi, \eta \in E \, , \quad b \in B \, ,
\]
and a $B$-valued inner product given, on simple tensors, by
\begin{equation}
\label{eq:interprod}
\inn{\xi_1 \otimes \eta_1, \xi_2 \otimes \eta_2} = \inn{\eta_1, \phi(\inn{\xi_1,\xi_2}) \eta_2}  
\end{equation}
and extended by linearity. The inner product is well defined and has all required properties; in particular, the null space 
$N=\{\zeta \in E\ot_{\rm alg} E \, ;\,  \inn{\zeta, \zeta} = 0\}$ 
is shown to coincide with the subspace generated by elements of the form in \eqref{ns}. 
One takes $E \ot_\phi E := E\ot_{\rm alg}E / N$ and defines $E \hot_\phi E$ to be the Hilbert module obtained by completing with respect to the norm induced by \eqref{eq:interprod}. 
The construction can be iterated and, for $n>0$, we denote by $E^{\hot_{\phi} n}$, the $n$-fold interior tensor power of $E$ over $B$. Like-wise, $(E^*)^{\hot_{\phi^*} n}$ denotes the $n$-fold interior tensor power of $E^*$ over $B$.

To lighten notation, in the following we define, for each $n \in \zz$, the modules
\[
E^{(n)} := \begin{cases} E^{\hot_{\phi} n} & n>0 \\
B & n=0 \\
(E^*)^{\hot_{\phi^*} (-n)} & n<0 
\end{cases} \, \, .
\]
Clearly, $E^{(1)}=E$ and $E^{(-1)}=E^*$.
We define the Hilbert $C^*$-module over $B$:
\[
E_{\infty}:= \bigoplus_{n \in \zz} E^{(n)} \, .
\]
For each $\xi \in E$ we have a bounded adjointable operator $S_\xi : E_\infty \to E_\infty$ defined component-wise by
\begin{align*}
S_\xi(b) & := \xi \cd b \, , & b \in B  \, , \\
S_\xi(\xi_1 \otimes \cdots\otimes \xi_n) & := \xi \ot \xi_1 \otimes \cdots\otimes \xi_n \, , & n>0 \, , \\
S_\xi(\la_{\xi_1} \otimes \cdots\otimes \la_{\xi_{-n}}) & := \la_{\xi_2 \cd \phi^{-1}(\te_{\xi_1,\xi})} \ot \la_{\xi_3} 
\otimes \cdots\otimes \la_{\xi_{-n}}  \, , & n<0  \, .
\end{align*}
In particular, $S_\xi(\la_{\xi_1}) = \phi^{-1}(\te_{\xi, \xi_1}) \in B$. 

\noindent
The adjoint of $S_\xi$ is easily found to be given by $S_{\la_\xi} := S_\xi^* : E_\infty \to E_\infty$:
\begin{align*}
S_{\la_\xi}(b) & := \la_\xi \cd b \, , & b \in B  \, ,  \\
S_{\la_\xi}(\xi_1 \olo \xi_n) & := \phi(\inn{\xi,\xi_1})(\xi_2) \ot \xi_3 \olo \xi_n \, , & n>0 \, , \\
S_{\la_\xi}(\la_{\xi_1} \olo \la_{\xi_{-n}}) & := \la_\xi \ot \la_{\xi_1} \olo \la_{\xi_{-n}} \, , & n<0 \, ;
\end{align*}
and in particular $S_{\la_\xi}(\xi_1) = \inn{\xi,\xi_1} \in B$. 

From its definition, each $E^{(n)}$ has a natural structure of Hilbert $C^*$-module over $B$ and,
with $\sK $ again denoting the Hilbert $C^*$-module compacts, we have isomorphisms
\[
\sK (E^{(n)}, E^{(m)}) \simeq E^{(m-n)} \, .
\] 

\begin{dfn}
The \emph{Pimsner algebra} of the pair $(\phi,E)$ is the smallest $C^*$-subalgebra of $\sL(E_\infty)$ 
which contains the operators $S_\xi : E_\infty \to E_\infty$ for all $\xi \in E$. The Pimsner algebra is denoted by $\C O_E$
with inclusion $\wit \phi : \C O_E \to \sL(E_\infty)$.
\end{dfn}

There is an injective $*$-homomorphism $i : B \to \C O_E$. This is induced by 
the injective $*$-homomorphism $\phi : B \to \sL(E_\infty)$ defined component-wise by
\[
\begin{split}
\phi(b)(b') & := b \cd b' \, , \\ 
\phi(b)(\xi_1\olo \xi_n) & := \phi(b)(\xi_1) \ot \xi_2 \olo \xi_n \, , \\
\phi(b)(\la_{\xi_1} \olo \la_{\xi_n}) & := \phi^*(b)(\la_{\xi_1}) \ot \la_{\xi_2} \olo \la_{\xi_n} = \la_{\xi_1 \cd b^*} \ot \la_{\xi_2} \olo \la_{\xi_n} \, ,
\end{split}
\]
and which factorizes through the Pimsner algebra $\C O_E \su \sL(E_\infty)$. 
Indeed, for all $\xi,\eta \in E$ it holds that $S_\xi S_{\eta}^* = i(\phi^{-1}(\te_{\xi,\eta}))$, that is the operator $S_\xi S_{\eta}^*$ on $E_\infty$ is right-multiplication by the element $\phi^{-1}(\te_{\xi,\eta})\in B$.

A Pimsner algebra is universal in the following sense \cite[Thm.~3.12]{Pim:CCC}:

\begin{theorem}\label{t:unipro}
Let $C$ be a $C^*$-algebra and let $\sigma : B \to C$ be a $*$-homomorphism. Suppose that there exist elements $s_\xi \in C$ for all $\xi \in E$ such that 
\begin{enumerate}
\item $\alpha s_\xi + \beta s_\eta = s_{\al \xi + \be \eta} $ \quad for all \, $\alpha, \beta \in \mathbb{C}$ and $\xi, \eta \in E$ ,\\
\item $s_{\xi} \sigma (b) =s_{\xi b}$ \ and \ $\sigma(b) s_{\xi} = s_{\phi(b)(\xi)}$ \quad for all \, $\xi \in E$ and $b \in B$ ,\\
\item $s^*_\xi s_\eta = \sigma (\inn{\xi, \eta})$ \quad for all \, $\xi, \eta \in E$ , \\
\item $s_\xi s_\eta^* = \si\big(\phi^{-1}(\te_{\xi,\eta}) \big)$ \quad for all \, $\xi, \eta \in E$ .
\end{enumerate}
Then there is a unique $*$-homomorphism $\wit{\sigma}: \mathcal{O}_E \rightarrow C$ with $\wit{\si}(S_\xi) = s_\xi$ for all $\xi \in E$.
\end{theorem}
Also, in the context of this theorem the identity $\wit{\sigma} \ci i = \si$ follows automatically.

\begin{remark}\label{re:ccorr}
In the paper \cite{Pim:CCC}, the pair $(\phi,E)$ was referred to as a \emph{Hilbert bimodule}, since the map
$\phi$ (taken to be injective there) naturally endows the right Hilbert module $E$ with a left module structure. 
As mentioned, our Assumption~\ref{a:hilmod} simplifies the construction to a great extent (see also \cite{AEE:MEC}).
For the pair $(\phi,E)$ with a general $*$-homomorphism $\phi : B \to \sL(E)$, 
(in particular, a non necessarily injective one), the name \emph{$C^*$-correspondence} over $B$ has recently 
emerged as a more common one, reserving the terminology Hilbert bimodule to the more restrictive case where one has both a left and a right inner product satisfying an extra compatibility relation.
\end{remark}

\subsection{Six term exact sequences}

With a Pimsner algebra there come two six term exact sequences in $KK$-theory. Firstly, 
since $\phi : B \to \sL(E)$ factorizes through the compacts $\sK(E) \su \sL(E)$, the following class is well defined.

\begin{dfn}\label{d:hilmodcla}
The class in $KK_0(B,B)$ defined by the even Kasparov module $(E, \phi, 0)$ (with trivial grading) will be denoted by $[E]$.
\end{dfn}

Next, let $P : E_\infty \to E_\infty$ denote the orthogonal projection with 
\[
\T{Im}(P) = \big( \op_{n = 1}^\infty E^{(n)} \big) \op B \su E_\infty \, .
\]
Notice that $[P,S_\xi] \in \sK(E_\infty)$ for all $\xi \in E$ and thus $[P,S] \in \sK(E_\infty)$ for all $S \in \C O_E$. 

Then, let $F := 2P-1 \in \sL(E_\infty)$ and recall that $\wit \phi : \C O_E \to \sL(E_\infty)$ is the inclusion.

\begin{dfn}
The class in $KK_1(\C O_E,B)$ defined by the odd Kasparov module $(E_\infty, \wit \phi, F)$ will be denoted by $[\pa]$.
\end{dfn}

For any separable $C^*$-algebra $C$ we then have the group homomorphisms
\[
[E] : KK_*(B,C) \to KK_*(B,C)\, , \q [E] : KK_*(C,B) \to KK_*(C,B) \\
\]
and
\[
[\pa] : KK_*(C,\C O_E) \to KK_{* + 1}(C,B) \, , \q [\pa] : KK_*(B,C) \to KK_{* + 1}(\C O_E,C) \, ,
\]
which are induced by the Kasparov product.

The six term exact sequences in $KK$-theory given in the following theorem were constructed by Pimsner, see \cite[Thm.~4.8]{Pim:CCC}.

\begin{theorem}\label{t:gysseq}
Let $\C O_E$ be the Pimsner algebra of the pair $(\phi,E)$ over the $C^*$-algebra $B$.
If $C$ is any separable $C^*$-algebra, there are two exact sequences: 
\[
\begin{CD}
KK_0(C, B) @>{1 - [E]}>> KK_0(C,B) @>{i_*}>> KK_0(C,\C O_E) \\
@A{[\pa]}AA & & @VV{[\pa]}V \\
KK_1(C,\C O_E) @<<{i_*}< KK_1(C,B) @<<{1 - [E]}< KK_1(C,B)
\end{CD}
\]
and \\
\[
\begin{CD}
KK_0(B,C) @<<{1 - [E]}< KK_0(B,C) @<<{i^*}< KK_0(\C O_E,C) \\
@VV{[\pa]}V & & @A{[\pa]}AA \\
KK_1(\C O_E,C) @>{i^*}>> KK_1(B,C) @>{1 - [E]}>> KK_1(B,C)
\end{CD}
\]

\noindent
with $i^*$, $i_*$ the homomorphisms in $KK$-theory induced by the inclusion $i: B \to \C O_E$.
\end{theorem}

\begin{remark}\label{rem:gysseq}
For  $C = \cc$, the first sequence above reduces to 
\[
\begin{CD}
K_0(B) @>{1 - [E]}>> K_0(B) @>{i_*}>> K_0(\C O_E) \\
@A{[\pa]}AA & & @VV{[\pa]}V \quad . \\
K_1(\C O_E) @<<{i_*}< K_1(B) @<<{1 - [E]}< K_1(B) 
\end{CD}
\]

This could be considered as a generalization of the classical \emph{Gysin sequence} in $K$-theory (see \cite[IV.1.13]{Ka78}) 
for the `line bundle' $E$ over the `noncommutative space' $B$ and with the map $1 - [E]$ having the role of the 
\emph{Euler class} $\chi(E):=1 - [E]$ of the line bundle $E$. The second sequence would then be an analogue in $K$-homology:
\[
\begin{CD}
K^0(B) @<<{1 - [E]}< K^0(B) @<<{i^*}< K^0(\C O_E) \\
@VV{[\pa]}V & & @A{[\pa]}AA \quad . \\
K^1(\C O_E) @>{i^*}>> K^1(B) @>{1 - [E]}>> K^1(B)
\end{CD}
\]
Examples of Gysin sequences in $K$-theory were given in \cite{ABL14} for line bundles over quantum projective spaces and leading to a class of quantum lens spaces. 
These examples will be generalized later on in the paper to a class of quantum lens spaces as circle bundles over quantum weighted projective spaces with arbitrary weights.
\end{remark}

\section{Pimsner algebras and circle actions}
An interesting source of Pimsner algebras consists of $C^*$-algebras which are equipped with a circle action and subject to an extra completeness condition on the associated spectral subspaces. We now investigate this relationship.

Throughout this section $A$ will be a $C^*$-algebra and $\{\si_z\}_{z \in S^1}$ will be a strongly continuous action of the circle $S^1$ 
on $A$.

\subsection{Algebras from actions}
For each $n \in \zz$, define the spectral subspace
\[
A_{(n)} := \big\{ \xi \in A \mid \si_z(\xi) = z^{-n} \, \xi \ \ \T{ for all } z \in S^1 \big\} \, .
\]
Then the invariant subspace $A_{(0)} \su A$ is a $C^*$-subalgebra and each $A_{(n)}$ is a (right) Hilbert $C^*$-module over $A_{(0)}$ with right action induced by the algebra structure on $A$ and $A_{(0)}$-valued inner product just $\inn{\xi,\eta} := \xi^* \, \eta$, for all $\xi,\eta \in A_{(n)}$.

\begin{assu}\label{a:semisat}
The data $(A, \si_z)$ as above is taken to satisfy the conditions:

\begin{enumerate}
\item The $C^*$-algebra $A_{(0)}$ is separable.
\item The Hilbert $C^*$-modules $A_{(1)}$ and $A_{(-1)}$ are full and countably generated over the $C^*$-algebra $A_{(0)}$. 
\end{enumerate}
\end{assu}

\begin{lemma}
With the $*$-homomorphism $\phi : A_{(0)} \to \sL(A_{(1)})$ simply defined by $\phi(a)(\xi) := a \, \xi$, the pair $(\phi,A_{(1)})$ satisfies the conditions of Assumption \ref{a:hilmod}.
\end{lemma}
\begin{proof}
To prove that $\phi : A_{(0)} \to \sL(A_{(1)})$ is injective, let $a \in A_{(0)}$ and suppose that $a \, \xi = 0$ for all $\xi \in A_{(1)}$. It then follows that $a \, \xi \, \eta^* = 0$ for all $\xi,\eta \in A_{(1)}$. But this implies that $a \, \inn{v,w} = 0$ for all $v,w \in A_{(-1)}$. Since $A_{(-1)}$ is full this shows that $a = 0$. We may thus conclude that $\phi : A_{(0)} \to \sL(A_{(1)})$ is injective, and the image of $\phi$ is therefore closed.

\noindent
To conclude that $\sK(A_{(1)}) \su \phi(A_{(0)})$ it is now enough to show that the operator $\te_{\xi,\eta} \in \phi(A_{(0)})$ 
for all $\xi,\eta \in A_{(1)}$. But this is clear since $\te_{\xi,\eta} = \phi(\xi \, \eta^*)$.

\noindent
To prove that $\phi(A_{(0)}) \su \sK(A_{(1)})$ it suffices to check that $\phi(\inn{v,w}) \in \sK(A_{(1)})$ for all $v,w \in A_{(-1)}$ (again since $A_{(-1)}$ is full). But this is true being $\phi(\inn{v,w}) = \te_{v^*,w^*}$.
\end{proof}

The condition that both $A_{(1)}$ and $A_{(-1)}$ are full over $A_{(0)}$ has the important consequence that the action $\{\si_z\}_{z \in S^1}$ is semi-saturated in the sense of the following:
\begin{dfn}
A circle action $\{\si_z\}_{z \in S^1}$ on a $C^*$-algebra $A$ is called \emph{semi-saturated} if $A$ is generated, as a $C^*$-algebra, by 
the fixed point algebra $A_{(0)}$ together with the first spectral subspace $A_{(1)}$.
\end{dfn}

\begin{prop}\label{p:semsat}
Suppose that $A_{(1)}$ and $A_{(-1)}$ are full over $A_{(0)}$. Then the circle action $\{\si_z\}_{z \in S^1}$ is semi-saturated.
\end{prop}
\begin{proof}
With $\T{cl}( \cd )$ refering to the norm-closure, we show that the Banach algebra 
\[
\T{cl}\Big( \sum_{n = 0}^\infty A_{(n)} \Big) \su A
\]
is generated by $A_{(1)}$ and $A_{(0)}$. A similar proof in turn shows that
\[
\T{cl}\Big( \sum_{n = 0}^\infty A_{(-n)} \Big) \su A
\]
is generated by $A_{(-1)}$ and $A_{(0)}$. Since the span $\sum_{n \in \zz} A_{(n)}$ is norm-dense in $A$ (see \cite[Prop.~2.5]{Exe:CPP}), this proves the proposition. We show by induction on $n \in \nn$ that
\[
(A_{(1)})^n := \M{span}\big\{ x_1 \clc x_n \mid x_1,\ldots, x_n \in A_{(1)} \big\}
\]
is dense in $A_{(n)}$. For $n = 1$ the statement is void.

Suppose thus that the statement holds for some $n \in \nn$.
Then, let $x \in A_{(n + 1)}$ and choose a countable approximate identity $\{ u_m \}_{m \in \nn}$ for the separable $C^*$-algebra $A_{(0)}$. Let $\ep > 0$ be given. We need to construct an element $y \in (A_{(1)})^{n + 1}$ such that
\[
\| x - y \| < \ep \ .
\]

To this end we first remark that the sequence $\{ x \cd u_m \}_{m \in \nn}$ converges to $x \in A_{(n + 1)}$. 
Indeed, this follows due to $x^* x \in A_{(0)}$ and since, for all $m \in \nn$, 
\[
\| x \cd u_m - x \|^2 = \| u_m x^* x u_m + x^* x - x^* x u_m - u_m x^* x \| \, .
\]
We may thus choose an $m \in \nn$ such that
\[
\| x \cd u_m - x \| < \ep/3 \ .
\]
Since $A_{(1)}$ is full over $A_{(0)}$, there are elements $\xi_1,\ldots,\xi_k$ and $\eta_1, \ldots, \eta_k \in A_{(1)}$ so that
\[
\| x \cd u_m - \sum_{j = 1}^k x \cd \xi_j^* \cd \eta_j \| < \ep / 3 \ .
\]
Furthermore, since $x \cd \xi_j^* \in A_{(n)}$ we may apply the induction hypothesis to find elements 
$z_1,\ldots,z_k \in (A_{(1)})^n$ such that
\[
\| \sum_{j = 1}^k x \cd \xi_j^* \cd \eta_j  - \sum_{j = 1}^k z_j \cd \eta_j \| < \ep/3 \ .
\]
Finally, it is straightforward to verify that for the element
\[
y := \sum_{j = 1}^k z_j \cd \eta_j \in (A_{(1)})^{n + 1}
\]
it holds that: $\| x - y \| < \ep$.  This proves the present proposition.
\end{proof}

Having a semi-saturated action one is lead to the following theorem \cite[Thm.~3.1]{AEE:MEC}.
\begin{theorem}\label{t:pimcir}
The Pimsner algebra $\C O_{A_{(1)}}$ is isomorphic to $A$. The isomorphism is given by $S_\xi \mapsto \xi$ for all $\xi \in A_{(1)}$.
\end{theorem}

\subsection{$\zz$-graded algebras}
In much of what follows, the $C^*$-algebras of interest with a circle action, will come from closures of dense $\zz$-graded $*$-algebras, with the $\zz$-grading defining the circle action in a natural fashion.

Let $\sA = \op_{n \in \zz} \sA_{(n)}$ be a $\zz$-graded unital $*$-algebra. The grading is compatible with the involution 
$^*$, this meaning that $x^* \in \sA_{(-n)}$ whenever $x \in \sA_{(n)}$ for some $n \in \zz$. 
For $w \in S^1$, define the $*$-automorphism $\si_w : \sA \to \sA$ by 
\[
\si_w : x \mapsto w^{-n} x \qquad \T{for} \quad  x \in \sA_{(n)} \, \quad n \in \zz \, .
\]
We will suppose that we have a $C^*$-norm $\| \cd \| : \sA \to [0,\infty)$ on $\sA$  satisfying
\[
\| \si_w(x) \| \leq \| x \| \q \T{for all } \quad w \in S^1 \, \quad x \in \sA \, ,
\]
thus the action has to be isometric. The completion of $\sA$ is denoted by $A$.

The following standard result is here for the sake of completeness and its use below. The proof relies on the existence of a conditional expectation naturally associated to the action.
\begin{lemma}\label{l:denspe}
The collection $\{ \si_w\}_{w \in S^1}$ extends by continuity to a strongly continuous action of $S^1$ on $A$. 
Each spectral subspace $A_{(n)}$ agrees with the closure of $\sA_{(n)} \su A$. 
\end{lemma}
\begin{proof}
Once $\sA_{(n)}$ is shown to be dense in $A_{(n)}$ the rest follows from standard arguments. 
Thus, for $n \in \zz$, define the bounded operator $E_{(n)} : A \to A_{(n)}$ by 
\[
E_{(n)} : x \mapsto \int_{S^1} w^n \, \si_w(x) \ \dd w \, ,
\]
where the integration is carried out with respect to the Haar-measure on $S^1$. We have that $E_{(n)}(x) = x$ for all $x \in A_{(n)}$ and then that $\| E_{(n)} \| \leq 1$. This implies that $\sA_{(n)} \su A_{(n)}$ is dense.
\end{proof}

Let now $d \in \nn$ and consider the unital $*$-subalgebra $\sA^{1/d} := \op_{n \in \zz} \sA_{(nd)} \su \sA$. Then $\sA^{1/d}$ is 
a $\zz$-graded unital $*$-algebra as well and we denote the associated circle action by $\si_w^{1/d} : \sA^{1/d} \to \sA^{1/d}$. 
Let $w \in S^1$ and choose a $z \in S^1$ such that $z^d = w$. Then
\[
\si_w^{1/d}(x_{nd}) = w^n \cd x_{nd} = z^{nd} \cd x_{nd} = \si_z(x_{nd})\, , \q\T{for all } \, x_{nd} \in \sA_{(nd)} \, , 
\]
and it follows that $\si_w^{1/d}(x) = \si_z(x)$ for all $x \in \sA^{1/d}$.  With the $C^*$-norm obtained by restriction 
$\|\cd \| : \sA^{1/d} \to [0,\infty)$, it follows in particular that
\[
\| \si_w^{1/d}(x) \| \leq \| x \|
\]
by our standing assumption on the compatibility of $\{ \si_w\}_{w \in S^1}$ with the norm on $\sA$.
The $C^*$-completion of $\sA^{1/d}$ is denoted by $A^{1/d}$.

\begin{prop}\label{p:isospesub}
Suppose that $\{\si_w\}_{w \in S^1}$ is semi-saturated on $A$ and let $d \in \nn$. Then we have unitary isomorphisms of Hilbert $C^*$-modules 
\[
(A_{(1)})^{\hot_{\phi} d} \simeq (A^{1/d})_{(1)} \q \M{and} \q (A_{(-1)})^{\hot_{\phi} d} \simeq (A^{1/d})_{(-1)}
\]
induced by the product $\psi: x_1 \olo x_d \mapsto x_1 \clc x_d$ . 
\end{prop}
\begin{proof}
We only  consider the case of $A_{(1)}$ since the the proof for $A_{(-1)}$ is the same.

Observe firstly that $(\sA^{1/d})_{(1)} = \sA_{(d)}$. Thus Lemma \ref{l:denspe} yields $A_{(d)} = (A^{1/d})_{(1)}$. This implies that the product $\psi : (\sA_{(1)})^{ \ot_{\sA_{(0)}} d} \to (\sA^{1/d})_{(1)}$ is a well-defined homomorphism of right modules over $\sA_{(0)}$ (here ``$\ot_{\sA_{(0)}}$'' refers to the algebraic tensor product of bimodules over $\sA_{(0)}$). Furthermore, since
\[
\inn{x_1 \olo x_d,y_1 \olo y_d} = x_d^* \clc x_1^* \cd y_1 \clc y_d \, ,
\]
we get that $\psi$ extends to a homomorphism $\psi : (A_{(1)})^{\hot_{\phi} d} \to A_{(1)}^{1/d}$ of Hilbert $C^*$-modules over $A_{(0)}$ with $\inn{ \psi(\xi), \psi(\eta)} = \inn{\xi,\eta}$ for all $\xi,\eta \in (A_{(1)})^{\hot_{\phi} d}$. 

It is therefore enough to show that $\T{Im}(\psi) \su (A^{1/d})_{(1)}$ is dense. But this is a consequence of \cite[Prop.~4.8]{Exe:CPP}.
\end{proof}

\begin{lemma}\label{l:semisateq}
Suppose that $\{ \si_w\}_{w \in S^1}$ satisfies the conditions of Assumption \ref{a:semisat}. Then $\{\si_w^{1/d}\}_{w \in S^1}$ satisfies the conditions of Assumption \ref{a:semisat} for all $d \in \nn$.
\end{lemma}
\begin{proof}
We only need to show that the Hilbert $C^*$-modules $A_{(d)}$ and $A_{(-d)}$ are full and countably generated over $A_{(0)}$.

By Proposition \ref{p:semsat} we have that $\{\si_w\}_{w \in S^1}$ is semi-saturated. It thus follows from Proposition \ref{p:isospesub} that
\begin{equation}\label{eq:isospesub}
A_{(d)} \simeq (A_{(1)})^{\hot_\phi d} \q \T{and} \q
A_{(-d)} \simeq (A_{(-1)})^{\hot_\phi d} \ .
\end{equation}
Since both $A_{(1)}$ and $A_{(-1)}$ are full and countably generated by assumption these unitary isomorphisms prove the lemma.
\end{proof}

The following result is a stronger version of Theorem~\ref{t:pimcir} since it incorporates all the spectral subspaces and not just the first one.

\begin{theorem}\label{t:pimcirsub}
Suppose that the circle action $\{\si_w\}_{w \in S^1}$ on $A$ satisfies the conditions in Assumption \ref{a:semisat}. Then the Pimsner algebra 
$\C O_{A_{(d)}} \simeq \C O_{(A_{(1)})^{\hot d}}$ is isomorphic to the $C^*$-algebra $A^{1/d}$ for all $d \in \nn$. The isomorphism is given by $S_\xi \mapsto \xi$ for all $\xi \in A_{(d)}$.
\end{theorem}
\begin{proof}
This follows by combining Lemma \ref{l:semisateq}, Proposition \ref{p:isospesub} and Theorem~\ref{t:pimcir}.
\end{proof}

We finally investigate what happens when the $C^*$-norm on $\sA = \op_{n \in \zz} \sA_{(n)}$ is changed. Thus, let $\| \cd \|' : \sA \to [0,\infty)$ be an alternative $C^*$-norm on $\sA$ satisfying
\[
\| \si_w(x) \|' \leq \| x \|' \q \T{for all } \, w \in S^1 \T{ and } x \in \sA \, .
\]
The corresponding completion $A'$ will carry an induced circle action $\{ \si_w'\}_{w \in S^1}$. 
The next theorem can be seen as a manifestation of the gauge-invariant uniqueness theorem, \cite[Thm. 6.2 and Thm. 6.4]{Ka04}. This property was indirectly used already in \cite[Thm. 3.12]{Pim:CCC} for the proof of the universal properties of Pimsner algebras.

\begin{theorem}\label{t:norequpim}
Suppose that $\| x \| = \| x \|'$ for all $x \in \sA_{(0)}$. Then $\{ \si_w\}_{w \in S^1}$ satisfies the conditions of Assumption \ref{a:semisat} if and only if $\{\si_w'\}$ satisfies the conditions of Assumption \ref{a:semisat}. And in this case, the identity map $\sA \to \sA$ induces an isomorphism $A \to A'$ of $C^*$-algebras. In particular, we have that $\| x \| = \| x \|'$ for all $x \in \sA$.
\end{theorem}
\begin{proof}
Remark first that the identity map $\sA_{(n)} \to \sA_{(n)}$ induces an isometric isomorphism of Hilbert $C^*$-modules $A_{(n)} \to A_{(n)}'$ for all $n \in \zz$. This is a consequence of the identity $\| x \| = \| x \|'$ for all $x \in \sA_{(0)}$. But then we also have isomorphisms
\[
(A_{(1)})^{\hot_\phi n} \simeq (A_{(1)}')^{\hot_\phi n} \q \T{and} \q (A_{(-1)})^{\hot_\phi n} \simeq (A_{(-1)}')^{\hot_\phi n}
\]
for all $n \in \nn$. These observations imply that $\{ \si_w\}_{w \in S^1}$ satisfies the conditions of Assumption \ref{a:semisat} if and only if $\{\si_w'\}$ satisfies the conditions of Assumption \ref{a:semisat}. But it then follows from Theorem~\ref{t:pimcir} that
\[
A \simeq \C O_{A_{(1)}} \simeq \C O_{A_{(1)}'} \simeq A' \, ,
\]
with corresponding isomorphism $A \simeq A'$ induced by the identity map $\sA \to \sA$.
\end{proof}

\section{Quantum principal bundles and $\zz$-graded algebras}
We start by recalling the definition of a quantum principal $U(1)$-bundle. 

Later on in the paper we shall exhibit a novel class of quantum lens spaces as principal 
$U(1)$-bundles over quantum weighted projective lines with arbitrary weights.

\subsection{Quantum principal bundles}
Define the unital complex algebra 
$$
\C O(U(1)) := \cc[z,z^{-1}]/\inn{1 - z z^{-1}}
$$ 
where $\inn{1 - z z^{-1}}$ denotes the ideal generated by $1 - z z^{-1}$ in the polynomial algebra $\cc[z,z^{-1}]$ in two variables. The algebra $\C O(U(1))$ is a Hopf algebra by defining, for all $n \in \zz$,  coproduct $\De : z^n \mapsto z^n \ot z^n$,  antipode $S : z^n \mapsto z^{-n}$ and counit $\ep : z^n \mapsto 1$. We simply write $\C O(U(1)) = \big( \C O(U(1)), \De, S, \ep\big)$ for short. 

Let $\sA$ be a complex unital algebra and suppose in addition that it is a right comodule algebra over $\C O(U(1))$, that is 
we have a homomorphism of unital algebras
\[
\De_R : \sA \to \sA \ot \C O(U(1)) \, ,
\]
which also provides a coaction of the Hopf algebra $\C O(U(1))$ on $\sA$. 

Let $\sB := \{ x \in \sA \mid \De_R(x) = x \ot 1 \}$ denote the unital subalgebra of $\sA$ consisting of coinvariant elements for the coaction.

\begin{dfn}\label{d:quapri}
One says that the datum $\big( \sA, \C O(U(1)), \sB \big)$ is a \emph{quantum principal $U(1)$-bundle} when the \emph{canonical map}
\[
\T{can} : \sA \ot_{\sB} \sA \to \sA \ot \C O(U(1)) \, , \q x \ot y \mapsto x \cd \De_R(y) \, ,
\]
is an isomorphism.
\end{dfn}

\begin{remark}
One ought to qualify Definition \ref{d:quapri}  by saying that the quantum principal bundle is `for the universal differential calculus'
\cite{ BrMa:QGQ}.
In fact, the definition above means that the right comodule algebra $\sA$ is a \emph{$\sB$-Galois extension}, and this is equivalent (in the present context) by \cite[Prop.~1.6]{Haj:SCQ} to the bundle being a quantum principal bundle for the universal differential calculus. 
\end{remark}

\subsection{Relation with $\zz$-graded algebras}\label{se:rzga}
We now provide a detailed analysis of the case where the quantum principal bundle structure comes from a $\zz$-grading of the `total space' algebra. This will lead to an alternative characterization of quantum $U(1)$-principal bundles in this setting. While this description is not new
(see for instance \cite[Lemma 5.1]{SiVe:RDN}), it is certainly more manageable. In particular, we will apply it in \S\ref{se:qls} below for the case of quantum lens spaces as $U(1)$-principal bundles over quantum weighted projective lines. 

Let $\sA = \op_{n \in \zz} \sA_{(n)}$ be a $\zz$-graded unital algebra and let $\C O(U(1))$ be the Hopf algebra defined in the previous section. Define the unital algebra homomorphism
\[
\De_R : \sA \to \sA \ot \C O(U(1)) \q x \mapsto x \ot z^{-n} \, , \, \, \, \T{for} \, \, \, x \in \sA_{(n)} \, .
\]
It is then clear that $\De_R$ turns $\sA$ into a right comodule algebra over $\C O(U(1))$. The unital subalgebra of coinvariant elements coincides with $\sA_{(0)}$.

\begin{theorem}\label{t:quapriseq}
The triple $\big( \sA, \C O(U(1)), \sA_{(0)} \big)$ is a quantum principal $U(1)$-bundle if and only if there exist finite sequences
\[
\{ \xi_j \}_{j = 1}^N \, , \, \, \{\be_i\}_{i = 1}^M \mbox{ in } \sA_{(1)} \q \M{and} \q
\{\eta_j\}_{j = 1}^N \, , \, \, \{\al_i\}_{i = 1}^M \mbox{ in } \sA_{(-1)}
\]
such that there hold identities:
\[
\sum_{j = 1}^N \xi_j \eta_j = 1_{\sA} = \sum_{i = 1}^M \al_i \be_i \, .
\]
\end{theorem}
\begin{proof}
Suppose first that $\big( \sA, \C O(U(1)), \sA_{(0)} \big)$ is a quantum principal $U(1)$-bundle. Thus, that the canonical map
\[
\T{can} : \sA \ot_{\sA_{(0)}} \sA \to \sA \ot \C O(U(1))
\]
is an isomorphism.
For each $n \in \zz$, define the idempotents
\[
\begin{array}{ll}
P_{(n)} : \C O(U(1)) \to \C O(U(1)) \, , \q & P_{(n)} : z^m \mapsto \de_{nm} z^m \q \mbox{and} \\
\\
E_{(n)} : \sA \to \sA \, , \q & E_{(n)} : x_m \mapsto \de_{nm} x_m 
\end{array}
\]
where $x_m \in \sA_{(m)}$ and where $\de_{nm} \in \{0,1\}$ denotes the Kronecker delta.
Clearly, 
\begin{equation}\label{eq:cancom}
\T{can} \ci (1 \ot E_{(-n)}) = (1 \ot P_{(n)}) \ci \T{can} : \sA \ot_{\sA_{(0)}} \sA \to \sA \ot \C O(U(1)) \, .
\end{equation}
for all $n \in \zz$. Let us now define the element 
\[
\ga := \T{can}^{-1}(1_{\sA} \ot z) = \sum_{j = 1}^N \ga^0_j \ot \ga_j^1 \, .
\]
It then follows from \eqref{eq:cancom} that
\[
\ga = (1 \ot E_{(-1)})(\ga) = \sum_{j = 1}^N \ga^0_j \ot E_{(-1)}(\ga_j^1)
\]
To continue, we remark that
\[
m(\ga) = m \ci \T{can}^{-1}(1_{\sA} \ot z) = (\T{id} \ot \ep)(1_{\sA} \ot z) = 1_{\sA}
\]
where $m : \sA \ot_{\sA_{(0)}} \sA \to \sA$ is the algebra multiplication. And this implies that
\[
1_{\sA} = \sum_{j = 1} \ga^0_j \cd E_{(-1)}(\ga_j^1) = \sum_{j = 1}^N E_{(1)}(\ga^0_j) \cd E_{(-1)}(\ga_j^1) \, .
\]
We therefore put, 
\[
\xi_j := E_{(1)}(\ga_0^j) \q \T{and} \q \eta_j := E_{(-1)}(\ga_1^j) \, ,  \q \T{for all} \, \, \,  j = 1,\ldots,N \, .
\]

Next, we define the element
\[
\de := \T{can}^{-1}(1_{\sA} \ot z^{-1}) = \sum_{i = 1}^M \de^0_i \ot \de_i^1 \, .
\]
An argument similar to the one before then shows that $\sum_{i = 1}^M \al_i \cd \be_i = 1_{\sA}$, with 
\[
\al_i := E_{(-1)}(\de^0_i) \q \T{and} \q \be_i := E_{(1)}(\de^1_i) \, ,  \q \T{for all} \, \, \,  i = 1,\ldots,M \ .
\]
This proves the first half of the theorem.

To prove the second half we suppose that there exist sequences 
$\{ \xi_j \}_{j = 1}^N$, $\{\be_i\}_{i = 1}^M $ in $\sA_{(1)}$ and 
$\{\eta_j\}_{j = 1}^N$, $\{\al_i\}_{i = 1}^M $in $\sA_{(-1)}$
such that $\sum_{j = 1}^N \xi_j \eta_j = 1_{\sA} = \sum_{i = 1}^M \al_i \be_i$. 

We then define the map $\M{can}^{-1} : \sA \ot \C O(U(1)) \to \sA \ot_{\sA_{(0)}} \sA$ by the formula
\[
\M{can}^{-1} : x \ot z^n \mapsto 
\begin{cases}
\sum_{J \in \{1,\ldots,N\}^n} \, x \, \xi_{j_1} \clc \xi_{j_n} \ot \eta_{j_n} \clc \eta_{j_1} \, ,
& \M{for} \,\,\, n \geq 0  \\ \\
\sum_{I \in \{1,\ldots,M\}^{-n}} \, x \, \al_{i_1} \clc \al_{i_{-n}} \ot \be_{i_{-n}} \clc \be_{i_1} \, ,
& \M{for} \,\,\,n \leq 0
\end{cases} \,\, .
\]
It is then straightforward to check that
\[
\M{can}^{-1} \ci \M{can} = \T{id} \q \M{and} \q
\M{can} \ci \M{can}^{-1} = \T{id} \, .
\]
This ends the proof of the theorem.
\end{proof}

\begin{remark}
The above theorem shows that $\big( \sA, \C O(U(1)), \sA_{(0)} \big)$ is a quantum principal $U(1)$-bundle if and only if $\sA$ is \emph{strongly} $\zz$-graded, see \cite[Lem.~I.3.2]{NaVO82}. Our next corollary is thus a consequence of \cite[Cor.~I.3.3]{NaVO82}. We present a proof here since we need the explicit form of the idempotents later on.
\end{remark}

\begin{cor}\label{co:mnfpm}
With the same conditions as in Theorem \ref{t:quapriseq}. 
The right-modules $\sA_{(1)}$ and $\sA_{(-1)}$ are finitely generated and projective over $\sA_{(0)}$.
\end{cor}
\begin{proof}
With the $\xi$'s and the $\eta$'s as above, define the module homomorphisms
\[
\begin{split}
\Phi_{(1)}  & : \sA_{(1)}  \to (\sA_{(0)})^N \, , \q \Phi_{(1)}(\zeta) = \ma{c}{\eta_1 \, \zeta \\ \eta_2 \, \zeta \\ \vdots \\ \eta_N \, \zeta} 
\q \T{and} \\
\Psi_{(1)}  & : (\sA_{(0)})^N \to \sA_{(1)}  \, , \q \Psi_{(1)}\ma{c}{x_1 \\ x_2 \\ \vdots \\ x_N} = \xi_1 \, x_1 + \xi_2 \, x_2 + \cdots + 
\xi_N \, x_N  \, .
\end{split}
\]
It then follows that $\Psi_{(1)} \Phi_{(1)} = \T{id}_{\sA_{(1)} }$. 
Thus $E_{(1)} := \Phi_{(1)}\Psi_{(1)}$ is an idempotent in $M_N(\sA_{(0)})$ and this proves the first half of the corollary.

Similarly, 
with the $\alpha$'s and the $\beta$'s as above, define the module homomorphisms
\[
\begin{split}
\Phi_{(-1)} & : \sA_{(-1)}  \to \C O(W_q(k,l))^2 \, , \q \Phi_{(-1)}(\zeta) = \ma{c}{\beta_1 \, \zeta \\ \beta_2 \, \zeta \\ \vdots 
\\ \beta_M \, \zeta} \q \T{and} \\
\Psi_{(-1)}  & : \C O(W_q(k,l))^2 \to \sA_{(-1)}  \, , \q \Psi_{(-1)}\ma{c}{x_1 \\ x_2 \\ \vdots \\ x_M} = \alpha_1 \, x_1 + \alpha_2 \, x_2 
+ \cdots + \alpha_M \, x_M  \, .
\end{split}
\]
Now one gets $\Psi_{(-1)}\Phi_{(-1)}= \T{id}_{\sA_{(-1)} }$. 
Thus $E_{(-1)}:= \Phi_{(-1)}\Psi_{(-1)}$ is an idempotent in $M_M(\sA_{(0)})$ as well. 
This finishes the proof of the corollary.
\end{proof}

Let $d \in \nn$ and consider the $\zz$-graded unital $\cc$-algebra $\sA^{1/d} := \op_{n \in \zz} \sA_{(dn)}$. 

As a consequence of Theorem \ref{t:quapriseq} we obtain the following:

\begin{prop}\label{p:quaprifin}
Suppose $\big( \sA, \C O(U(1)), \sA_{(0)} \big)$ is a quantum principal $U(1)$-bundle. Then $\big( \sA^{1/d}, \C O(U(1)), \sA_{(0)} \big)$ is a quantum principal $U(1)$-bundle for all $d \in \nn$.
\end{prop}
\begin{proof}
Let the finite sequences $\{ \xi_j\}_{j = 1}^N$, $\{\be_i\}_{i = 1}^M$ in $\sA_{(1)}$ and $\{\eta_j\}_{j = 1}^N$, $\{\al_i\}_{i = 1}^M$ in $\sA_{(-1)}$ be as in Theorem \ref{t:quapriseq}. 
For each multi-index $J \in \{1,\ldots,N\}^d$ and each multi-index $I \in \{1,\ldots,M\}^d$ define the elements
\[
\begin{split}
& \xi_J := \xi_{j_1} \clc \xi_{j_d} \, , \q  \be_I := \be_{i_d} \clc \be_{i_1} \in \sA_{(d)} \q \M{and} \\
& \eta_J := \eta_{j_d} \clc \eta_{j_1} \, , \q \al_I := \al_{i_1} \clc \al_{i_d} \in \sA_{(-d)} \, .
\end{split}
\]
It is then clear that
\[
\sum_{J \in \{1,\ldots,N\}^d } \xi_J \, \eta_J = 1_{\sA^{1/d}} =
\sum_{I \in \{1,\ldots,M\}^d } \al_I \, \be_I \ .
\]
This proves the proposition by an application of Theorem \ref{t:quapriseq}.
\end{proof}

Remark that it follows from Proposition \ref{p:quaprifin} and Corollary \ref{co:mnfpm} that when $\big( \sA, \C O(U(1)), \sA_{(0)} \big)$ is a quantum principal bundle then the right modules $\sA_{(d)}$ and $\sA_{(-d)}$ are finitely generated projective over $\sA_{(0)}$ for all $d \in \nn$.

\section{Quantum weighted projective lines}
We recall the definition of the quantum weighted projective lines as fixed point algebras of circle actions on the quantum $3$-sphere. These algebras play the role of the coordinate functions on the base space which parametrizes the lines generating the quantum lens spaces (as total spaces). Corresponding $C^*$-algebras will be the analogues of continuous functions on the base and total space respectively. The latter $C^*$-algebra will be given as a Pimsner algebra coming from the line bundles.

\subsection{Coordinate algebras}\label{ss:cooalgwei}
Let $n \in \nn_0$ and let $q \in (0,1)$.

\begin{dfn}
The coordinate algebra $\C O(S_q^{2n+ 1})$ of the \emph{quantum sphere} $S_q^{2n+1}$ is the universal unital $*$-algebra with generators $z_0,\ldots,z_n$ and relations
\[
\begin{split}
& z_i z_j = q z_j z_i \, \, \T{ for } \, i < j \, \, , \q z_i z_j^* = q z_j^* z_i \, \, \T{ for } \, i \neq j  \, , \\
& z_i z_i^* = z_i^* z_i + (q^{-2} - 1) \sum_{m = i + 1}^n z_m z_m^* \, \, , \q \sum_{m = 0}^n z_m z_m^* = 1 \, .
\end{split}
\]
\end{dfn}
\noindent 
This algebra was introduced in \cite{VS91}. 
Next, let $L = (l_0,\ldots,l_n) \in \nn^{n + 1}$ be fixed. We then have a circle action $\{ \si_w^L\}_{w \in S^1}$ on $\C O(S_q^{2n+1})$ defined on generators by
\[
\si_w^L : z_i \mapsto w^{l_i} z_i \q \T{for all } \ i \in \{0,\ldots,n\} \, .
\]

\begin{dfn}
The coordinate algebra $\C O(W_q(L))$ of the \emph{quantum weighted projective space} $W_q(L)$ is the fixed point algebra of the circle action $\{\si_w^L\}_{w \in S^1}$. Thus
\[
\C O(W_q(L)) := \big\{ x \in \C O(S_q^{2n+ 1}) \mid \si_w^L(x) = x \M{ for all } w \in S^1 \big\} \, .
\]
\end{dfn}

From now on, we will suppose that $n = 1$ and that $k := l_0$ and $l := l_1$ are coprime. By \cite[Thm.~2.1]{BrFa:QT}, the algebraic quantum projective line $\C O(W_q(k,l))$ agrees with the unital $*$-subalgebra of $\C O(S_q^3)$ generated by the elements $z_0^l (z_1^*)^k$ and $z_1 z_1^*$. Alternatively, one may identify $\C O(W_q(k,l))$ with the universal unital $*$-algebra with generators $a,b$, subject to the relations
\begin{align*}
& b^* = b \, \, , \q b a = q^{-2l} \ ab \,\, , \\
& a a^* = q^{2kl} \ b^k \prod_{m = 0}^{l - 1} (1 - q^{2m} b) \,\, , \q a^* a = b^k \prod_{m = 1}^l (1 - q^{-2m} b) \, . 
\end{align*}
The identification is just $a \mapsto z_0^l (z_1^*)^k$ and $b \mapsto z_1 z_1^*$ (we have exchanged the names of generators with respect to \cite{BrFa:QT}). In particular $\C O(W_q(1,1))=\C O(\mathbb{C}P_q^1)$, while $\C O(W_q(1,l))$ was named 
\emph{quantum teardrop} in  \cite{BrFa:QT}.

\subsection{$C^*$-completions}\label{ss:teacom}
We fix $k,l \in \nn$ to be coprime positive integers. 

\begin{dfn}\label{de:cqwps}
The algebra of continuous functions on the \emph{quantum weighted projective line} $W_q(k,l)$ is the universal enveloping $C^*$-algebra, denoted 
$C(W_q(k,l))$, of the coordinate algebra $\C O(W_q(k,l))$. 
\end{dfn}

Let $\sK$ denote the $C^*$-algebra of compact operators on the separable Hilbert space $l^2(\nn_0)$ of all square summable sequences indexed by $\nn_0$, with orthonormal basis $\{ e_p \}_{p \in \nn_0} $. It was shown in \cite[Prop.~5.1]{BrFa:QT} that $C(W_q(k,l))$ is isomorphic to the unital $C^*$-algebra
\[
\wit{\op_{s = 1}^l \sK} \su \sL\big( \op_{s = 1}^l l^2(\nn_0) \big) \, , 
\]
where $~\wit{\cd}~$ denotes the unitalization functor. The isomorphism is induced by the direct sum of representations $\op_{s = 1}^l \pi_s : \C O(W_q(k,l)) \to \sL\big(\op_{s = 1}^l l^2(\nn_0) \big)$ where each $\pi_s$ is defined on generators by
\begin{equation}\label{eq:defreptea}
\begin{split}
& \pi_s(z_1 z_1^*)(e_p) := q^{2s} \ q^{2 lp} \ e_p \, \, \, , \q \pi_s(z_0^l (z_1^*)^k) (e_0) := 0 \, , \\
& \pi_s(z_0^l (z_1^*)^k)(e_p) := q^{k (lp + s)} \ \prod_{m = 1}^l (1 - q^{2(lp + s -m)})^{1/2} \ e_{p - 1} \, \, \, , \, \, p \geq 1  \, .
\end{split}
\end{equation}

Notice that the $C^*$-algebra $C(W_q(k,l))$ does not depend on $k$. 
As a consequence one has the following corollary due to Brzezi\'nski and Fairfax, see \cite[Cor.~5.3]{BrFa:QT}.
\begin{cor}\label{c:ktwps}
The $K$-groups of $C(W_q(k,l))$ are:
\[
K_0(C(W_q(k,l))) = \zz^{l+1} \, , \q K_1(C(W_q(k,l))) = 0 \, .
\]
\end{cor}

Notice that the $K$-theory groups of the quantum weighted projective lines do not agree with the $K$-theory groups of their commutative counterparts: In the commutative case, the $K_0$-group is given by $K_0(C(W(k,l))) = \zz^{2}$ independently of both weights $k$ and $l$, see \cite[Prop.~2.5]{AA94}.

\begin{dfn}
The algebra of continuous functions on the \emph{quantum $3$-sphere} $S_q^3$ is the universal enveloping $C^*$-algebra,   
$C(S_q^3)$, of the coordinate algebra $\C O(S_q^3)$.  
\end{dfn}

The (weighted) circle action $\big\{ \si_w^{(k,l)} \big\}_{w \in S^1}$ on $\C O(S_q^3)$ will be denoted  simply by $\{ \si_w \}_{w \in S^1}$. 
It induces a strongly continuous circle action on $C(S_q^3)$. We let $C(S_q^3)_{(0)}$ denote the fixed point algebra of this action.

\begin{lemma}\label{l:norequtea}
The inclusion $\C O(W_q(k,l)) \su \C O(S_q^3)$ induces an isomorphism of unital $C^*$-algebras,
\[
i : C(W_q(k,l)) \to C(S_q^3)_{(0)} \, .
\]
\end{lemma}
\begin{proof} 
Clearly, one has $\T{Im}(i) \su C(S_q^3)_{(0)}$ 
and $\T{Im}(i)$ is dense by the argument used in the proof of Lemma \ref{l:denspe}.

It therefore suffices to show that $i : C(W_q(k,l)) \to C(S_q^3)$ is injective. To this end, consider the $*$-homomorphism $\pi := \op_{s = 1}^l \pi_s : \C O(W_q(k,l)) \to \sL\big( \op_{s = 1}^l l^2(\nn_0) \big)$. Then, by \cite[Prop.~2.4]{BrFa:QT} there exist a $*$-homomorphism $\rho : \C O(S_q^3) \to \sL(l^2(\nn_0))$ and an isomorphism $\phi : \sL\big( \op_{s = 1}^l l^2(\nn_0) \big) \to \sL(l^2(\nn_0))$ such that 
\[
\phi \ci \pi = \rho \ci i : \C O(W_q(k,l)) \to \sL(l^2(\nn_0)) \, .
\]
Let now $x \in \C O(W_q(k,l))$. It follows from the above, that
\[
\|x\| = \| \pi(x) \| = \| (\phi \ci \pi)(x) \| = \| (\rho \ci i)(x) \| \leq \| i(x) \| \, .
\]
This proves that $i : C(W_q(k,l)) \to C(S_q^3)_{(0)}$ is an isometry and it is therefore injective.
\end{proof}

Let $\sL^1$ denotes the trace class operators on the Hilbert space $l^2(\nn_0)$.

\begin{lemma}\label{l:factraide}
The $*$-homomorphism $\pi := \op_{s = 1}^l \pi_s : \C O(W_q(k,l)) \to \wit{\op_{s = 1}^l \sK}$ factorizes through the unital $*$-subalgebra $\wit{\op_{s = 1}^l \sL^1}$.
\end{lemma}
\begin{proof}
Let $s \in \{1,\ldots,l\}$. We only need to show that $\pi_s( z_0^l (z_1^*)^k ), \pi_s(z_1 z_1^*) \in \sL^1$. 

With notation $a := z_0^l (z_1^*)^k$ and $b := z_1 z_1^*$, the operator $\pi_s(b) : l^2(\nn_0) \to l^2(\nn_0)$ is positive and diagonal with eigenvalues $\{ q^{2s} \ q^{2lp} \}_{p = 0}^\infty$ each of multiplicity $1$.

It is immediate to show that $\pi_s(b)^{1/2} \in \sL^1$. Indeed, from 
\eqref{eq:defreptea}, 
\[
\T{Tr}(\pi_s(b)^{1/2}) = \sum_{p = 0}^\infty \ q^s \ q^{lp} = q^s \ (1 - q^l)^{-1} < \infty \, ,
\]
having restricted the deformation parameter to $q \in (0,1)$.
From $\pi_s(b)^{1/2} \in \sL^1$ the inclusion $\pi_s(b) \in \sL^1$ follows as well.

To obtain that $\pi_s(a) \in \sL^1$ we need to verify that $|\pi_s(a)| \in \sL^1$. Now, recall  that 
\[
a^* a = b^k \cd \prod_{m = 1}^l (1 - q^{-2m} b) \, .
\]
Using this relation, we may compute the absolute value: 
\[
|\pi_s(a)| = \pi_s(b)^{k/2} \cd \Big( \prod_{m = 1}^l (1 - q^{-2m} \pi_s(b)) \Big)^{1/2} \, .
\]
Since $\sL^1$ is an ideal in $\sL(l^2(\nn_0))$ we may thus conclude that $|\pi_s(a)| \in \sL^1$.
\end{proof}

\section{Quantum lens spaces}\label{se:qls}
We define $3$-dimensional quantum lens spaces $\C O\big( L_q(dlk;k,l) \big)$ as fixed point algebras for the action of a finite cyclic group on the coordinate algebra of the quantum $3$-sphere. We show that these spaces are quantum principal bundles over quantum weighted projective spaces. Our examples are more general than those of \cite{BrFa:QT}. 
As said the enveloping $C^*$-algebras of the lens spaces will be given as Pimsner algebras.

\subsection{Coordinate algebras}\label{ss:cooalg}
Let $k,l \in \nn$ be coprime positive integers. For each $d \in \nn$ define the action of the cyclic group $\zz/ (dlk) \zz$ 
on the quantum sphere $S_q^3$, 
\[
\al^{1/d} : \zz/ (dlk) \zz \ti \C O(S_q^3)  \to \C O(S_q^3) \, ,
\]
by letting on generators:
\begin{equation}\label{actzz}
\al^{1/d}(1,z_0) := \exp ( \frac{2 \pi \ii}{ dl } ) \ z_0 \q \T{and} \q \al^{1/d}(1,z_1) := \exp( \frac{2 \pi \ii}{ dk} ) \ z_1 \, .
\end{equation}
\begin{dfn}
The coordinate algebra for the \emph{quantum lens space} $L_q(dlk; k,l)$ is the fixed point algebra of the action $\al^{1/d}$. This unital $*$-algebra is denoted by $\C O\big( L_q(dlk;k,l) \big)$. Thus
\[
\C O\big( L_q(dlk;k,l) \big) := \big\{ x \in \C O(S_q^3) \mid \al^{1/d}(1,x) = x \big\} \, .
\]  
\end{dfn}
The elements $z_0^l (z_1^*)^k$ and $z_1 z_1^*$, generating the weighted projective space algebra $\C O(W_q(k,l))$, 
are clearly invariant leading, for any $d\in \nn$, to an algebra inclusion
\[
\C O(W_q(k,l)) \hookrightarrow \C O\big( L_q(dlk;k,l) \big) \, . 
\]
Next, for each $n \in \nn_0$, consider the subspaces of $\C O(S_q^3)$ given by
\begin{equation}\label{n-subsp}
\begin{split}
\sA_{(n)}(k,l) & := \sum_{j = 0}^n \ (z_0^*)^{l j} (z_1^*)^{k (n -j)} \cd \C O(W_q(k,l)) \, , \\
\sA_{(-n)}(k,l) & := \sum_{j = 0}^n \ (z_0)^{l j} (z_1)^{k (n -j)} \cd \C O(W_q(k,l)) \,.
\end{split}
\end{equation}
By construction these subspaces are in fact right-modules over $\C O(W_q(k,l))$. 

Recall that the algebra $\C O(S_q^3)$ admits \cite{Wor88} a vector space basis given by the vectors 
$\{ e_{p,r,s} \mid p \in \zz, \, r, s \in \nn_0 \}$, where
\[
e_{p,r,s} = 
\begin{cases}
z_0^p z_1^r (z_1^*)^s & \T{for}  \,\,\,\,\, p \geq 0 \\
(z_0^*)^{-p} z_1^r (z_1^*)^s & \T{for}  \,\,\,\,\, p \leq 0
\end{cases} \,\, .
\]
\begin{lemma}\label{l:chaspesub}
Let $n \in \zz$. It holds that
\[
\begin{split}
e_{p,r,s} \in \sA_{(n)}(k,l) & \lrar p k + (r -s) l = -nkl \\ 
& \lrar \si_w^{k,l}(e_{p,r,s}) = w^{-nkl} e_{p,r,s} \, , \, \, \forall \, w \in S^1 \, .
\end{split}
\]
As a consequence, it holds that
\[
x \in \sA_{(n)}(k,l) \lrar \si_w^{k,l}(x) = w^{-nkl} x \, , \, \, \forall \, w \in S^1 \, .
\]
\end{lemma}
\begin{proof}
Clearly one has that
\[
\begin{split}
e_{p,r,s} \in \sA_{(n)}(k,l) & \rar p k + (r -s) l = - nkl  \\
& \lrar \si_w^{k,l}(e_{p,r,s}) = w^{-nkl} e_{p,r,s} \, , \, \, \forall \, w \in S^1 \, .
\end{split}
\]

Thus, it only remains to prove the implication
\[
p k + (r - s) l = -nkl \rar e_{p,r,s} \in \sA_{(n)}(k,l) \, .
\]
Then, suppose $p k + (r - s) l = -nkl$. Since $k,l \in \nn$ are coprime there exists integers $d_0,d_1 \in \zz$ such that 
$p = d_0 l$ and $(r - s) = d_1 k$. Furthermore, $d_0 + d_1 = - n$. \\
Suppose first that $(r - s) \, , \, \, p \geq 0$. Then,
\[
e_{p,r,s} = z_0^p z_1^{(r -s)} (z_1 z_1^*)^s = z_0^{l d_0} z_1^{k d_1} (z_1 z_1^*)^s \in \sA_{(-d_0 - d_1)}(k,l) = \sA_{(n)}(k,l) \, .
\]
Suppose next that $p \geq 0$ and $(r - s) \leq 0$. Then, 
\[
e_{p,r,s} = z_0^p (z_1^*)^{s - r} (z_1 z_1^*)^r = z_0^{l d_0} (z_1^*)^{-d_1 k} (z_1 z_1^*)^r \, .
\]
We now have two sub-cases: Either $d_0 \geq - d_1$ or $-d_1 \geq d_0$. When $d_0 \geq - d_1$, it follows from the above that
\[
e_{p,r,s} = z_0^{l (d_0 + d_1)} z_0^{-d_1 l} (z_1^*)^{-d_1 k} (z_1 z_1^*)^r \in \sA_{(n)}(k,l) \, .
\]
On the other hand, if $-d_1 \geq d_0$, we have that
\[
e_{p,r,s} = z_0^{l d_0} (z_1^*)^{k d_0} (z_1^*)^{(-d_1 - d_0) k} (z_1 z_1^*)^r  \in \sA_{(n)}(k,l) \, .
\]
The remaining two cases (when $p \leq 0$ and $(r -s) \geq 0$ and when $p \, , \, \, (r - s) \leq 0$) follow by 
similar arguments. This proves the lemma.
\end{proof}

\begin{prop}\label{p:zetgralen}
The subspaces $\big\{ \sA_{(dn)}(k,l) \big\}_{n \in \zz}$ gives $\C O(L_q(dlk;k,l))$ the structure of a $\zz$-graded unital $*$-algebra.
\end{prop}
\begin{proof}
We need to prove that the vector space sum provides a bijection
\[
\op_{n \in \zz} \sA_{(dn)}(k,l) \to \C O(L_q(dlk;k,l)) \, .
\]

Suppose thus that $\sum_{n \in \zz} x_n = 0$ where $x_n \in \sA_{(dn)}(k,l)$ for all $n \in \zz$ and $x_n = 0$ for all but finitely many $n \in \zz$. It then follows from Lemma \ref{l:chaspesub} that the terms $x_n$ lie in different homogeneous spaces for the circle action $\big\{ \si_w^{k,l} \big\}_{w \in S^1}$ on $\C O(S_q^3)$. We may then conclude that $x_n = 0$ for all $n \in \zz$. This proves the claimed injectivity.

Next, let $x \in \C O(L_q(dlk;k,l))$. Without loss of generality we may take $x = e_{p,r,s}$ for some $p \in \zz$ and $r,s \in \nn_0$. The fact that $x \in \C O(L_q(dlk;k,l))$ then means that
\[
p/(dl) + (r -s)/(dk) \in \zz \ \lrar\ p k + (r - s) l \in (dkl) \ \zz
\]
It then follows from Lemma~\ref{l:chaspesub} that $e_{p,r,s} \in \sum_{n \in \zz} \sA_{(dn)}(k,l)$. This proves surjectivity.

Finally, let $x \in \sA_{(dn)}(k,l)$ and $y \in \sA_{(dm)}(k,l)$. It only remains to prove that $x^* \in \sA_{(-dn)}(k,l)$ and 
$x \, y \in \sA_{(d(n+m))}(k,l)$. But these properties also follow immediately from 
Lemma~\ref{l:chaspesub} since $\si_w^{k,l}$ is a $*$-automorphism of $\C O(S_q^3)$ for each $w \in S^1$.
\end{proof}

\subsection{Lens spaces as quantum principal bundles}
The right-modules $\sA_{(1)}(k,l)$ and $\sA_{(-1)}(k,l)$ play a central role. Recall from \eqref{n-subsp} that they are given by
\[
\begin{split}
\sA_{(1)}(k,l) & := (z_1^*)^{k} \cd \C O(W_q(k,l)) + (z_0^*)^l \cd \C O(W_q(k,l))  \q \M{and} \\
\sA_{(-1)}(k,l) & := z_1^k \cd \C O(W_q(k,l)) + z_0^l \cd \C O(W_q(k,l)) \, .
\end{split} 
\]

\begin{prop}\label{p:fingenpro}
There exist elements
\[
\xi_1,\xi_2,\be_1,\be_2 \in \sA_{(1)}(k,l) \q \M{and} \q
\eta_1, \eta_2, \al_1,\al_2 \in \sA_{(-1)}(k,l)
\]
such that
\[
\xi_1 \eta_1 + \xi_2 \eta_2 = 1 = \al_1 \be_1 + \al_2 \be_2
\]
\end{prop}
\begin{proof}
Firstly, a repeated use of the defining relations of the algebra $\C O(S_q^3)$ leads to
\[
(z_0^*)^l z_0^l = \prod_{m = 1}^l (1 - q^{-2m} z_1 z_1^*) \, .
\]
Then, define the polynomial $F \in \cc[X]$ by the formula
\[
F(X) := \Big( 1 - \prod_{m = 1}^l (1 - q^{-2m} X) \Big)/X \, .
\]
Since $z_1 z_1^* = z_1^* z_1$ one has that
\[
(z_0^*)^l z_0^l + z_1^* \ F(z_1 z_1^*) \ z_1 = 1 \, .
\]
In particular, this implies that
\[
\begin{split}
1 & = \big( (z_0^*)^l z_0^l + z_1^* \ F(z_1 z_1^*) \ z_1 \big)^k
= \sum_{j = 0}^k \big( (z_0^*)^l z_0^l \big)^j \ \big( z_1^* \ F(z_1 z_1^*) \ z_1 \big)^{k-j} \ {k \choose j} \\
& = (z_1^*)^k  \big( F(z_1 z_1^*) \big)^k z_1^k + \sum_{j = 1}^k  \big( (z_0^*)^l z_0^l \big)^j \ 
\big( 1 - (z_0^*)^l z_0^l \big)^{k - j} \ {k \choose j} \\
& = (z_1^*)^k  \big( F(z_1 z_1^*) \big)^k z_1^k + 
(z_0^*)^l \left\{ \sum_{j = 1}^k  \big( z_0^l (z_0^*)^l \big)^{j-1} \big( 1 - z_0^l (z_0^*)^l \big)^{k - j} \ {k \choose j} \right\} z_0^l \, .
\end{split}
\]
Define now the polynomial $G \in \cc[X]$ by the formula
\begin{equation}\label{eq:geedef}
G(X) := (1 - (1 - X)^k )/X =\sum_{j=1}^k X^{j-1} (1-X)^{k-j} \ {k \choose j} \, ,
\end{equation}
so that
\[
\sum_{j = 1}^k  \big( z_0^l (z_0^*)^l \big)^{j-1} \big( 1 - z_0^l (z_0^*)^l \big)^{k - j} \ {k \choose j} = G\big( z_0^l (z_0^*)^l \big) \, .
\]
And this enables us to write the above identities as\\
\begin{equation}\label{eq:promod}
1 = (z_1^*)^k  \big( F(z_1 z_1^*) \big)^k z_1^k + (z_0^*)^l G\big(z_0^l (z_0^*)^l \big) z_0^l \, .
\end{equation}

\noindent
Notice that both $F(z_1 z_1^*)$ and $G\big(z_0^l (z_0^*)^l \big)$ belong to $\C O(W_q(k,l))$. We thus define
\begin{align*}
& \xi_1 := (z_1^*)^k  \big( F(z_1 z_1^*) \big)^k \, , \qquad  \eta_1 := z_1^k \, ,  \\
& \xi_2 := (z_0^*)^l \, G\big(z_0^l (z_0^*)^l \big)  \, , \, \qquad  \eta_2 := z_0^l  
\end{align*}
and this proves the first half of the proposition.
\medskip

\noindent 
To prove the second half, we consider instead the identity
\[
z_0^l (z_0^*)^l = \prod_{m = 0}^{l-1} (1 - q^{2m} z_1^* z_1) \, ,
\]
which again follows by a repeated use of the defining identities for $\C O(S_q^3)$.

The polynomial $\wit F \in \cc[X]$ is now given by the formula
\[
\wit F(X) := \Big( 1 - \prod_{m = 0}^{l-1} (1 - q^{2m} X) \Big)/X \, .
\]
and we obtain that
\[
z_0^l (z_0^*)^l + z_1\tilde{ F}(z_1 z_1^*) z_1^* = 1 \, .
\]
By taking $k^{\T{th}}$ powers and computing as above, this yields that
\[
\begin{split}
1 & = z_1^k  \big( \tilde{F}(z_1 z_1^*) \big)^k (z_1^*)^k + 
z_0^l \left\{ \sum_{j = 1}^k  {k \choose j} \big( (z_0^*)^l z_0^l \big)^{j-1} \big( 1 - (z_0^*)^l z_0^l \big)^{k - j} \right\} (z_0^*)^l \, .
\end{split}
\]
This identity may be rewritten as
\[
1 = z_1^k  \big( \wit F(z_1 z_1^*) \big)^k (z_1^*)^k + z_0^l G\big((z_0^*)^l z_0^l \big) (z_0^*)^l \, ,
\]
where $G \in \cc[X]$ is again the one defined by \eqref{eq:geedef}. \\
Since both $\wit{F}(z_1 z_1^*)$ and $G\big( (z_0^*)^l z_0^l \big)$ belong to $\C O(W_q(k,l))$ we define
\begin{align*}
& \al_1 := z_1^k  \big( \wit F(z_1 z_1^*) \big)^k \, , \qquad  \be_1 := (z_1^*)^k  \, , \\
& \al_2 := z_0^l \, G\big((z_0^*)^l z_0^l \big)  \, , \,  \qquad \be_2 := (z_0^*)^l \, . 
\end{align*}
This ends the proof of the present proposition.
\end{proof} 

The next proposition is now an immediate consequence of Proposition~\ref{p:zetgralen}, Proposition~\ref{p:fingenpro}, Theorem~\ref{t:quapriseq}, and Proposition~\ref{p:quaprifin}.

\begin{prop} 
The triple $\big( \C O(L_q(dlk); k,l), \C O(U(1)), \C O(W_q(k,l))\big)$ is a quantum principal $U(1)$-bundle for each $d \in \nn$.
\end{prop}

\subsection{$C^*$-completions}
We fix $k,l \in \nn$ to be coprime positive integers. 
Let $d \in \nn$. With $C(S_q^3)$ the $C^*$-algebra of continuous functions on the quantum sphere $S_q^3$, 
the action of the cyclic group $\zz/ (dlk) \zz$ given on generators in \eqref{actzz} results into an action
\[
\al^{1/d} : \zz/(dkl) \zz \ti C(S_q^3) \to C(S_q^3) \, .
\]

\begin{dfn}
The $C^*$-algebra of continuous functions on the \emph{quantum lens space} $L_q(dlk;k,l)$ is the fixed point algebra of this action. 
It is denoted by $C(S_q^3)^{1/d}$. 
Thus
\[
C(S_q^3)^{1/d} := \big\{ x \in C(S_q^3) \mid \al^{1/d}(1,x) = x \big\} \, .
\]  
\end{dfn}

\begin{lemma}
The $C^*$-quantum lens space  $C(S_q^3)^{1/d}$
is the closure of the algebraic quantum lens space $\C O(L_q(dkl;k,l))$ with respect to the universal $C^*$-norm on $\C O(S_q^3)$.
\end{lemma}
\begin{proof}
This follows by applying the bounded operator $E_{1/d} : C(S_q^3) \to C(S_q^3)^{1/d}$, 
\[
E_{1/d} : x \mapsto \frac{1}{dkl}\sum_{m = 1}^{dkl} \al^{1/d}([m],x) \, ,
\]
with $[m]$ denoting the residual class in $\zz/(dkl) \zz$ of the integer $m$. 
\end{proof}

Alternatively, and in parallel with Definition~\ref{de:cqwps}, we could define the $C^*$-quantum lens space as the universal enveloping $C^*$-algebra of the algebraic quantum lens space $\C O(L_q(dkl;k,l))$. We will denote this $C^*$-algebra by $C(L_q(dkl;k,l))$.
\begin{lemma}
For all $d \in \nn$, the identity map $\C O(L_q(dkl;k,l)) \to \C O(L_q(dkl;k,l))$ induces an isomorphisms of $C^*$-algebras,
\[
C(S_q^3)^{1/d} \simeq C(L_q(dkl;k,l)) \, .
\]
\begin{proof}
We use Theorem~\ref{t:norequpim}. Indeed, let $d \in \nn$ and let $\| \cd \| : \C O(S_q^3) \to [0,\infty)$ and $\| \cd \|' : \C O(L_q(dkl;k,l)) \to [0,\infty)$ denote the universal $C^*$-norms of the two different unital $*$-algebras in question. We then have $\| x \| \leq \| x \|'$ for all $x \in \C O(L_q(dkl;k,l))$ since the inclusion $\C O(L_q(dkl;k,l)) \to \C O(S_q^3)$ induce a $*$-homomorphism $C(L_q(dkl;k,l)) \to C(S_q^3)^{1/d}$. But we also have $\| x \|' \leq \| x \|$ since the restriction $\| \cd \| : \C O(W_q(k,l)) \to [0,\infty)$ is the maximal $C^*$-norm on $\C O(W_q(k,l))$ by Lemma~\ref{l:norequtea}.
\end{proof}
\end{lemma}

From now on, to lighten the notation, denote by $B := C(W_q(k,l))$ the $C^*$-quantum weighted projective line. Furthermore, let $E$ denote the Hilbert $C^*$-module over $B$ obtained as the closure of the module $\sA_{(1)}(k,l)$ in the universal $C^*$-norm on the quantum sphere $\C O(S_q^3)$. As usual, we let $\phi : B \to \sL(E)$ denote the $*$-homomorphism induced by the left multiplication $B \ti C(S_q^3) \to C(S_q^3)$.

We are ready to realize the $C^*$-quantum lens spaces as Pimsner algebras.
\begin{theorem}\label{t:pimlenspa}
For all $d \in \nn$, there is an isomorphism of $C^*$-algebras,
\[
\C O_{E^{\hot_\phi d}} \simeq C(S_q^3)^{1/d} \, ,
\]
given by 
\[
S_{\xi_1 \olo \xi_d} \mapsto \xi_1 \clc \xi_d \quad \T{for all} \quad \xi_1,\ldots,\xi_d \in E \, . 
\]
\end{theorem}
\begin{proof}
Recall from Proposition~\ref{p:zetgralen} that, for all $d \in \nn$, it holds that
\[
\C O(L_q(dlk;k,l)) \simeq \op_{n \in \zz} \sA_{(dn)}(k,l) \, . 
\]

Let us denote by $\{\rho_w\}_{w \in S^1}$ the associated circle action on $\C O(L_q(dlk;k,l))$. Then, we have $\|\rho_w(x) \| \leq \|x \|$ for all $x \in \C O(L_q(dlk;k,l))$ and all $w \in S^1$, where $\| \cd \|$ is the norm on $C(S_q^3)^{1/d}$ (the restriction of the maximal $C^*$-norm on $C(S_q^3)$). To see this, choose a $z \in S^1$ such that $z^{dkl} = w$. Then $\si_z^{(k,l)}(x) = \rho_w(x)$, where the weighted circle action $\si^{(k,l)} : S^1 \ti C(S_q^3) \to C(S_q^3)$ is the one defined at the beginning of \S\ref{ss:cooalgwei}.
 
An application of Theorem~\ref{t:pimcirsub} now shows that $\C O_{E^{\hot_\phi d}} \simeq C(S_q^3)^{1/d}$ for all $d \in \nn$, provided that $\{\rho_w\}_{w \in S^1}$ satisfies the conditions of Assumption~\ref{a:semisat}. To this end,  
taking into account the analysis of the coordinate algebra $\C O(L_q(lk;k,l))$ provided in \S\ref{ss:cooalg}, the only non-trivial thing to check is that the collections 
\[
\inn{E,E} := \T{span}\big\{ \xi^* \eta \mid \xi, \eta \in E \big\} \q \T{and} \q \inn{E^*,E^*} := \T{span}\big\{ \xi \eta^* \mid \xi,\eta \in E \big\}
\]
are dense in $C(W_q(k,l))$. But this follows at once from Proposition~\ref{p:fingenpro}.
\end{proof}

\section{$KK$-theory of quantum lens spaces}\label{s:kktlen}

We now combine the results obtained until this point and, using methods coming from the Pimsner algebra constructions, we are able to compute the $KK$-theory of the quantum lens spaces $L_q(dkl;k,l)$ for any coprime $k,l \in \nn$ and any $d \in \nn$. 

As before we let $E$ denote the Hilbert $C^*$-module over the quantum weighted projective line $C(W_q(k,l))$ which is obtained as the closure of $\sA_{(1)}(k,l)$ in $C(S_q^3)$.

The two polynomials in $\C O(W_q(k,l))$ in the proof of Proposition~\ref{p:fingenpro}, written as
\[
\begin{split}
(F(z_1 z_1^*))^k & = \Big( \big(1 - (z_0^*)^l z_0^l \big)/ (z_1 z_1^*) \Big)^k \q \T{and} \\
G\big( z_0^l (z_0^*)^l \big) & = \big(1 - (1 - z_0^l (z_0^*)^l)^k \big)/ (z_0^l (z_0^*)^l) \, ,
\end{split}
\]
are manifestly positive, since $\| z_1 z_1^*\| \leq 1$ and thus also $\| z_0^l (z_0^*)^l \|, \| (z_0^*)^l z_0^l \| \leq 1$ in $C(W_q(k,l))$. 
Thus it makes sense to take their square roots:
\[
\begin{split}
\xi_1 & := F(z_1 z_1^*)^{k/2} = \Big( \big(1 - (z_0^*)^l z_0^l \big)/ (z_1 z_1^*) \Big)^{k/2} \in C(W_q(k,l)) \q \T{and} \\
\xi_0 & := G\big( z_0^l (z_0^*)^l \big)^{1/2} = \Big( \big(1 - (1 - z_0^l (z_0^*)^l)^k \big)/ (z_0^l (z_0^*)^l)\Big)^{1/2} \in C(W_q(k,l)) \, .
\end{split}
\]

Next, define the morphism of Hilbert $C^*$-modules $\Psi : E \to C(W_q(k,l))^2$ by
\[
\Psi : \eta \mapsto \ma{c}{\xi_1 z_1^k \ \eta \\ \xi_0 z_0^l \ \eta} \, ,
\]
whose adjoint $\Psi^* : C(W_q(k,l))^2 \to E$ is then given by 
\[
\Psi^* : \ma{c}{x \\ y} \mapsto (z_1^*)^k \xi_1 \ x + (z_0^*)^l \xi_0 \ y \, .
\]
It then follows from  \eqref{eq:promod} that $\Psi^* \Psi = \T{id}_E$. 
The associated orthogonal projection is
\begin{equation}\label{e:bp}
P := \Psi \Psi^* = \ma{cc}{ \xi_1 \ (z_1 z_1^*)^k \ \xi_1 & \xi_1 \ z_1^k (z_0^*)^l \ \xi_0 \\ 
\xi_0 \ z_0^l (z_1^*)^k \ \xi_1 & \xi_0 \ z_0^l (z_0^*)^l \ \xi_0 } \in M_2(C(W_q(k,l))) \, . 
\end{equation}

\subsection{Fredholm modules over quantum weighted projective lines}
We recall \cite[Chap.~IV]{Co94} that an \emph{even Fredholm module} over a $*$-algebra $\sA$ is a datum $(H, \rho, F, \gamma)$ where $H$ is a Hilbert space of a representation $\rho$ of $\sA$, the operator $F$ on $H$ is such that $F^2=F$ and $F^2= 1$, with a $\zz/2\zz$-grading $\gamma$, $\gamma^2=1$, which commutes with the representation and such that $\gamma F + F \gamma = 0$. Finally, for all $a\in\sA$ the commutator $[F, \rho(a)]$ is required to be compact. The Fredholm module is said to be \emph{$1$-summable} if the commutator $[F, \rho(a)]$ is trace class for all $a \in \sA$. 

Now, the quantum sphere $S_q^3$ is the `underlying manifold' of the quantum group $\mathrm{SU}_q(2)$.
The latter's counit when restricted to the subalgebra $\C O(W_q(k,l))$ yields a one-dimensional representation 
$\ep : \C O(W_q(k,l)) \to \cc$, simply given on generators by,
\[
\ep(z_1 z_1^*)= \ep(z_0^l (z_1^*)^k):= 0 \, , \quad \ep(1) = 1 \, .
\]
Next, let $H := l^2(\nn_0)\otimes \cc^2$. We use the subscripts ``$+$'' and ``$-$'' to indicate that the corresponding spaces are thought of as being even or odd respectively, for a $\zz/2\zz$-grading $\gamma$:$H_{\pm}$ will be two copies of $H$. 
For each $s \in \{1,\ldots,l\}$, with the $*$-representation $\pi_s $ given in \eqref{eq:defreptea}, define the even $*$-homomorphism 
\[
\rho_s : \C O(W_q(k,l)) \to \sL\big(H_+ \op H_- \big) \, ,
\quad \rho_s : x \mapsto \ma{cc}{\pi_s(\Psi x \Psi^*) & 0 \\ 0 & \ep(\Psi x \Psi^*) } \, .
\]
We are slightly abusing notation here: the element $\Psi x \Psi^*$ is a $2\times2$ matrix, hence $\pi_s $ and $\varepsilon$ have to be applied component-wise. Next, define
\begin{equation}\label{e:Fg}
F=\ma{cc}{0 & 1 \\ 1 & 0} \,, \quad \gamma = \ma{cc}{1 & 0 \\ 0 & -1} \, .
\end{equation}

\begin{lemma}\label{l:fmf}
The datum $\sF_s := \big(H_+ \op H_-, \rho_s, F, \gamma \big), 
$ defines an even $1$-summable Fredholm module over the coordinate algebra $\C O(W_q(k,l))$.
\end{lemma}
\begin{proof}
It is enough to check that $\pi_s(\Psi z_1 z_1^* \Psi^*), \pi_s(\Psi z_0^l (z_1^*)^k \Psi^*) \in \sL^1(H)$ 
and furthermore that $\pi_s(P) - \ep(P) \in \sL^1(H)$, for  $P$ the projection in \eqref{e:bp}.

That the two operators involving the generators $z_1 z_1^*$ and $z_0^l (z_1^*)^k$ lie in $\sL^1(H)$ 
follows easily from Lemma~\ref{l:factraide}.  
To see that $\pi_s(P) - \ep(P) \in \sL^1(H)$ note that
\[
\ep(P) = \ma{cc}{0 & 0 \\ 0 & 1} \, .
\]
The desired inclusion then follows since Lemma \ref{l:factraide} yields that the operators $\pi_s(z_1 z_1^*)^k$, $\pi_s(z_0^l (z_1^*)^k)$, and $\pi_s(1 - z_0^l (z_0^*)^l)$ are of trace class.
\end{proof}

For $s = 0$, we take
\[
\rho_0 := \ma{cc}{\ep & 0 \\ 0 & 0} : C(W_q(k,l)) \to \sL(\cc \op \cc) 
\] 
and define the even $1$-summable Fredholm module
\[
\sF_0 := \big( \cc_+ \op \cc_-, \rho_0, F, \gamma \big) \, .
\]

\begin{remark}
The $1$-summable $l+1$ Fredholm modules over $\C O(W_q(k,l))$ we have defined are different from the $1$-summable Fredholm modules defined in \cite[\S4]{BrFa:QT}. The present Fredholm modules are obtained by ``twisting'' the Fredholm modules in \cite{BrFa:QT} with the Hilbert $C^*$-module $E$.
\end{remark}

\subsection{Index pairings}
Recall the representations $\pi_s$ of $C(W_q(k,l))$ given in \eqref{eq:defreptea}.

For each $r \in \{1,\ldots,l\}$, let $p_r \in C(W_q(k,l))$ denote the orthogonal projection defined by the requirement
\begin{equation}\label{eq:protea}
\pi_s(p_r) = \fork{ccc}{ e_{00} & \T{for} & s = r \\ 0 & \T{for} & s \neq r }  \, ,
\end{equation}
where $e_{00} : l^2(\nn_0) \to l^2(\nn_0)$ denotes the orthogonal projection onto the closed subspace $\cc e_0 \su l^2(\nn_0)$.
For $r = 0$, let $p_0 = 1 \in C(W_q(k,l))$. The classes of these $l+1$ projections $\{p_r, \ r=0,1, \dots, l\}$ form a basis for the group 
$K_0(C(W_q(k,l)))$ given in Corollary~\ref{c:ktwps}.

On the other hand we have the classes in the $K$-homology group $K^0(C(W_q(k,l)))$ represented by the even $1$-summable Fredholm modules $\sF_s$, $s = 0,\ldots,l$, which we described in the previous paragraph.

We are interested in computing the index pairings 
\[
\inn{[\sF_s],[p_r]} := \tfrac{1}{2} \ \T{Tr}\big( \ga F[F,\rho_s(p_r)] \big) \in \zz \, , \q \T{for} \, \, \, r,s \in \{0,\ldots,l\} \, .
\]

\begin{prop}\label{p:indpaicom} 
It holds that:
\[
\inn{[\sF_s],[p_r]} = \fork{ccc}{1 & \T{for} & s = r \\ 1 & \T{for} & r = 0 \\ 0 & & \T{else} } \, .
\]
\end{prop}
\begin{proof}
Suppose first that $r,s \in \{1,\ldots,l\}$. We then have:
\[
\inn{[\sF_s],[p_r]} = \T{Tr}\big( \pi_s(\Psi p_r \Psi^*) \big) \, ,
\]
and the above operator trace is well-defined since $\pi_s(\Psi p_r \Psi^*)$ is an orthogonal projection in $M_2(\sK)$ and it is therefore of trace class. We may then compute as follows:
\[
\begin{split}
\T{Tr}\big( \pi_s(\Psi p_r \Psi^*) \big) 
& = \T{Tr}\big( \pi_s(\xi_1 z_1^k p_r (z_1^*)^k \xi_1 )\big)
+ \T{Tr}\big( \pi_s(\xi_0 z_0^l p_r (z_0^*)^l \xi_0 ) \big) \\
& = \T{Tr}\big( \pi_s(p_r (z_1^*)^k \xi_1^2 z_1^k ) \big)
+ \T{Tr}\big( \pi_s(p_r (z_0^*)^l \xi_0^2 z_0^l) \big) \\
& = \T{Tr}\big( \pi_s(p_r) \big)
= \de_{sr} \, ,
\end{split}
\]
where the second identity follows from \cite[Cor.~3.8]{Sim:TIA} and 
$\de_{sr} \in \{0,1\}$ denotes the Kronecker delta.

If $r \in \{1,\ldots,l\}$ and $s = 0$, then $\rho_0(p_r) = 0$ and thus $\inn{[\sF_0],[p_r]} = 0$.

Next, suppose that $r = s = 0$. Then
\[
\inn{[\sF_0],[p_0]} = \T{Tr}\ma{cc}{1 & 0 \\ 0 & 0} = 1 \, .
\]

Finally, suppose that $r = 0$ and $s \in \{1,\ldots,l\}$. We then compute
\[
\begin{split}
\inn{[\sF_s],[p_0]} & = \T{Tr}\big( \pi_s(P) - \ep(P) \big)
= \T{Tr}\big( \pi_s(\xi_1^2 (z_1 z_1^*)^k) \big) + \T{Tr}\big( \pi_s(\xi_0 z_0^l (z_0^*)^l \xi_0) - 1 \big) \\
& = \T{Tr}\big( \pi_s( 1 - (z_0^*)^l z_0^l)^k \big) - \T{Tr}\big( \pi_s(1 - z_0^l (z_0^*)^l)^k \big) \, .
\end{split}
\]
We will prove in the next lemma that this quantity is equal to $1$. This will complete the proof of the present proposition.
\end{proof}

\begin{lemma} It holds that:
\[
\T{Tr}\big( \pi_s( 1 - (z_0^*)^l z_0^l)^k \big) - \T{Tr}\big( \pi_s(1 - z_0^l (z_0^*)^l)^k \big)
= \T{Tr}\big(\pi_s([z_0^l,(z_0^*)^l])\big) 
= 1 \, .
\]
\end{lemma}
\begin{proof}
Notice firstly that $\pi_s\big( 1 - (z_0^*)^l z_0^l \big), \pi_s\big( 1 - z_0^l (z_0^*)^l \big) \in \sL^1(l^2(\nn_0))$ 
by Lemma~\ref{l:factraide}. It then follows by induction that
\[
\T{Tr}\big( \pi_s( 1 - (z_0^*)^l z_0^l)^k \big) - \T{Tr}\big( \pi_s(1 - z_0^l (z_0^*)^l)^k \big)
= \T{Tr}\big(\pi_s( [z_0^l,(z_0^*)^l])\big) \, .
\]
Indeed, with $x : = z_0^l$, for all $j \in \{2,3,\ldots\}$, one has that,
\[
\begin{split}
& \T{Tr}\big( \pi_s( 1 - x^* x)^j \big) - \T{Tr}\big( \pi_s(1 - x x^*)^j \big) \\
& \q = \T{Tr}\big( \pi_s( 1 - x^*x)^{j-1} \big)
- \T{Tr}\big( \pi_s( x x^* (1 - x x^*)^{j-1}) \big)
- \T{Tr}\big( \pi_s(1 - x x^*)^j \big) \\
& \q = \T{Tr}\big( \pi_s( 1 - x^* x )^{j-1} \big) - \T{Tr}\big( \pi_s(1 - x x^* )^{j-1} \big) \, .
\end{split}
\]
It therefore suffices to show that $\T{Tr}\big( \pi_s( [z_0^l, (z_0^*)^l] ) \big) = 1$. 
Now, one has:
\[
[z_0^l, (z_0^*)^l ] = \sum_{m = 0}^l (-1)^m q^{m(m-1)} {l \choose m}_{q^2} (1 - q^{-2ml}) \ (z_1 z_1^*)^m
\]
where the notation ${l \choose m}_{q^2}$ refers to the $q^2$-binomial coefficient, defined by the identity
\[
\prod_{m = 1}^l (1 + q^{2(m-1)} Y) = \sum_{m = 0}^l q^{m(m-1)} {l \choose m}_{q^2} Y^m 
\]
in the polynomial algebra $\cc[Y]$. Then, as in \cite[Prop.~4.3]{BrFa:QT} one computes:
\[
\begin{split}
\T{Tr}\big( \pi_s( [z_0^l, (z_0^*)^l] ) \big) & = \sum_{m = 1}^l (-1)^m q^{m(m-1)} {l \choose m}_{q^2} (1 - q^{-2ml}) 
\frac{q^{2ms}} {1-q^{2ml}} \\ 
& = 1 - \sum_{m = 0}^l (-1)^m q^{m(m-1)} {l \choose m}_{q^2} \ q^{2m(s-l)}
\\ 
& = 1 - \prod_{m = 1}^l (1 - q^{2(s-m)}) =1 \, ,
\end{split}
\]
since, due to $s \in \{1,\ldots,l\}$ one of the factors in the product must vanish.
\end{proof}

\begin{remark}
The non-vanishing of the pairings in Proposition~\ref{p:indpaicom} for $r=0$ means that the class of the projection $P$ in \eqref{e:bp} is non-trivial in $K_0(C(W_q(k,l)))$. (In this case the pairings are computing the couplings of the Fredholm modules of \cite[\S4]{BrFa:QT} with the projection $P$.)
Geometrically this means that the line bundle $\sA_{(1)}(k,l)$ over $\C O(W_q(k,l))$ and then the quantum principal $U(1)$-bundle $\C O(W_q(k,l)))\hookrightarrow \C O(L_q(dlk); k,l)$ are non-trivial.
\end{remark}

\subsection{Gysin sequences}
To ease the notation, we now let $C(W_q) := C(W_q(k,l))$ and $C(L_q(d)) := C(L_q(dkl;k,l))$. Also as before we let $E$ denote the Hilbert $C^*$-module over $C(W_q)$ obtained as the closure of $\sA_{(1)}(k,l)$ in $C(S_q^3)$. The $*$-homomorphism $\phi : C(W_q) \to \sL(E)$ is induced by the product on $C(S_q^3)$.

For each $d \in \nn$, let $[E^{\hot d}] \in KK(C(W_q),C(W_q))$ denote the class of the Hilbert $C^*$-module 
$E^{\hot_\phi d}$ 
as in Definition~\ref{d:hilmodcla}. And recall from Theorem~\ref{t:pimlenspa} that 
the Pimnser algebra $\C O_{E^{\hot_\phi d}}$ can be identified with $C(L_q(d))$:
\[
\C O_{E^{\hot_\phi d}} \simeq C(L_q(d)) \, .
\]
Then, given any separable $C^*$-algebra $B$, by
Theorem~\ref{t:gysseq} we obtain two six term exact sequences:
\begin{equation}\label{eq:sixexamodI-0}
\begin{CD}
KK_0(B, C(W_q)) @>{1 - [E^{\hot d}]}>> KK_0(B,C(W_q)) @>{i_*}>> KK_0\big(B,C(L_q(d))\big) \\
@A{[\pa]}AA & & @VV{[\pa]}V \\
KK_1(B,C(L_q(d))) @<<{i_*}< KK_1(B,C(W_q)) @<<{1 - [E^{\hot d}]}< KK_1(B,C(W_q))
\end{CD}
\end{equation}
and
\begin{equation}\label{eq:sixexamodII-0}
\begin{CD}
KK_0(C(W_q),B) @<<{1 - [E^{\hot d}]}< KK_0(C(W_q),B) @<<{i^*}< KK_0\big(C(L_q(d)),B\big) \\
@VV{[\pa]}V & & @A{[\pa]}AA \\
KK_1\big(C(L_q(d)),B\big) @>{i^*}>> KK_1(C(W_q),B) @>{1 - [E^{\hot d}]}>> KK_1(C(W_q),B)
\end{CD} \, .
\end{equation}
We will refer to these two sequences as the \emph{Gysin sequences} (in $KK$-theory) for the quantum lens space $L_q(dkl;k,l)$.

\begin{remark}\label{rem:gysseqII}
For  $B = \cc$, the first sequence above was first constructed in \cite{ABL14} for quantum lens spaces in any dimension $n$ (and not just for $n=1$) but with weights all equal to one; so that the `base space' was a quantum projective space. 
\end{remark}

\subsection{Computing the $KK$-theory of quantum lens spaces}

We recall from \cite[Prop.~5.1]{BrFa:QT} that $C(W_q)$ is isomorphic to $\wit{\sK^l}$ (see also \S\ref{ss:teacom}). In particular, this means that $C(W_q)$ is $KK$-equivalent to $\cc^{l + 1}$. 

To show this equivalence explicitly, for each $s \in \{0,\ldots,l\}$ we define a $KK$-class $[\Pi_s] \in KK(C(W_q),\cc)$  via the Kasparov module $\Pi_s\in \ee (C(W_q),\cc)$ given by:
\[
\Pi_s := \big(l^2(\nn_0)_+ \op l^2(\nn_0)_-, \wit{\pi}_s, F, \gamma \big) \quad \T{ for } \, s \neq 0 \quad \T{and} 
\quad \Pi_0 := ( \cc,\ep,0) \quad\T{ for } \,   s = 0 \, ,
\]
with $F$ and $\gamma$ the canonical operators in \eqref{e:Fg}. The representation is
\[
\wit{\pi}_s=\ma{cc}{\pi_s & 0 \\ 0 & \ep} \, ,
\]
with the representation $\pi_s$ given by \eqref{eq:defreptea} and $\ep$ is (induced by) the counit. 

Furthermore, for each $r \in \{0,\ldots,l\}$ we define the $KK$-class $[I_r] \in KK(\cc,C(W_q))$ by the Kasparov module
\[
I_r := \big( C(W_q), i_r, 0 \big) \in \ee (\cc,C(W_q)) \, ,
\]
where $i_r : \cc \to C(W_q)$ is the $*$-homomorphism defined by $i_r : 1 \mapsto p_r$ with the orthogonal projections $p_r \in C(W_q)$ given in \eqref{eq:protea}. 

Upon collecting these classes as   
\[
[\Pi] := \op_{s = 0}^l [\Pi_s] \in KK(C(W_q),\cc^{l + 1}) \quad \T{and} \quad [I] := \op_{r = 0}^l [I_r] \in KK(\cc^{l+1},C(W_q)) \, , 
\]
it follows that $[I] \hot_{C(W_q)} [\Pi] = [1_{\cc^{l+1}}]$ and that $[\Pi] \hot_{\cc^{l+1}} [I] = [1_{C(W_q)}]$,
from stability of $KK$-theory (see \cite[Cor.~17.8.8]{Bl}).  

We need a final tensoring with the Hilbert $C^*$-module $E$. This yields a class 
\[
[I_r] \hot_{C(W_q)} [E] \hot_{C(W_q)} [\Pi_s] \in KK(\cc,\cc) \, ,
\]
for each $s,r \in \{0,\ldots,l\}$. Then, we let $M_{sr} \in \zz$ denote the corresponding integer in $KK(\cc,\cc) \simeq \zz$,
with $M := \{ M_{sr} \}_{s,r = 0}^l \in M_{l + 1}(\zz)$ the corresponding matrix.

As a consequence the six term exact sequence in \eqref{eq:sixexamodI-0} becomes
\begin{equation}\label{eq:sixexamodI}
\begin{CD}
\op_{r =0}^l K^0(B) @>{1 - M^d}>> \op_{s = 0}^l K^0(B) @>>> KK_0\big(B,C(L_q(d))\big) \\
@AAA & & @VVV \\
KK_1(B,C(L_q(d))) @<<< \op_{s = 0}^l K^1(B) @<<{1 - M^d}< \op_{r = 0}^l K^1(B)
\end{CD}
\end{equation}
while, with $M^t \in M_{l + 1}(\zz)$ denoting the matrix transpose of $M \in M_{l + 1}(\zz)$, 
the six term exact sequence in \eqref{eq:sixexamodII-0} becomes
\begin{equation}\label{eq:sixexamodII}
\begin{CD}
\op_{s = 0}^l K_0(B) @<<{1 - (M^t)^d}< \op_{r = 0}^l K_0(B) @<<< KK_0\big(C(L_q(d)),B\big) \\
@VVV & & @AAA \\
KK_1\big(C(L_q(d)),B\big) @>>> \op_{r = 0}^l K_1(B) @>{1 - (M^t)^d}>> \op_{s = 0}^l K_1(B)
\end{CD} \, .
\end{equation}

\medskip
In order to proceed we therefore need to compute the matrix $M \in M_{l + 1}(\zz)$.

\begin{lemma}\label{l:kasprocom}
The Kasparov product $[E] \hot_{C(W_q)} [\Pi_s] \in KK(C(W_q),\cc)$ is represented by the Fredholm module $\sF_s$ 
in Lemma~\ref{l:fmf} for each $s \in \{0,\ldots,l\}$.
\end{lemma}
\begin{proof}
Recall firstly that the class $[E] \in KK(C(W_q),C(W_q))$ is represented by the Kasparov module
\[
\big( E, \phi , 0 \big) \in \ee (C(W_q),C(W_q)) \, , 
\]
where $\phi : C(W_q) \to \sL(E)$ is induced by the product on the algebra $C(S_q^3)$. 
It then follows from the observations in the beginning of \S\ref{s:kktlen} that $(E, \phi , 0)$ is equivalent to the Kasparov module
\[
\big( C(W_q)^2, \Psi \phi \Psi^*, 0 \big) \in \ee (C(W_q),C(W_q)) \, .
\]
Suppose next that $s = 0$. The Kasparov product $[E] \hot_{C(W_q)} [\Pi_0]$ is then represented by the Kasparov module
\[
\big( C(W_q)^2 \hot_\ep \cc, \Psi \phi \Psi^* \ot 1, 0 \big) \in \ee ( C(W_q),\cc) \, , 
\]
which is equivalent to the Kasparov module
\[
\big( \cc_+ \op \cc_- \, , \ma{cc}{\ep & 0 \\ 0 & 0}, \ma{cc}{0 & 1 \\ 1 & 0} \big) \, .
\]
This proves the claim of the lemma in this case.

Suppose thus that $s \in \{1,\ldots,l\}$. The Kasparov product $[E] \hot_{C(W_q)} [\Pi_s]$ is then represented by the Kasparov module given by the $\zz/2\zz$-graded Hilbert space
\[
\big( C(W_q)^2 \hot_{\pi_s} \, l^2(\nn_0) \big)_+ \op \big( C(W_q)^2 \hot_{\ep} \, l^2(\nn_0) \big)_- 
\simeq { H_+ \op H_-}
\]
with associated $*$-homomorphism 
\[
\rho_s = \ma{cc}{\pi_s(\Psi \phi \Psi^*) & 0 \\ 0 & \ep(\Psi \phi \Psi^*)} : C(W_q) \to \sL\big(H_+ \op H_-\big) \, ,
\]
and with Fredholm operator $F$ and grading $\gamma$ the canonical ones in \eqref{e:Fg}. 
This proves the claim of the lemma in these cases as well.
\end{proof}

The results of Lemma~\ref{l:kasprocom} and Proposition~\ref{p:indpaicom} now yield the following:

\begin{prop}\label{p:commat}
The matrix $M = \{ M_{sr} \} \in M_{l + 1}(\zz)$ has entries
\[
M_{sr} = \inn{[\sF_s],[I_r] } = \fork{ccc}{ 1 & \T{for} & s = r \\ 1 & \T{for} & r = 0 \\ 0 & & \T{else} } \, .
\]
\end{prop}

A combination of Proposition~\ref{p:commat} and the six term exact sequences in \eqref{eq:sixexamodI} and \eqref{eq:sixexamodII} then allows us to compute the $K$-theory and the $K$-homology of the quantum lens space $L_q(dlk;k,l)$ for all $d \in \nn$. 

When $B = \cc$, the sequence in \eqref{eq:sixexamodI} reduces to
\[
\xymatrix{ 
0 \longrightarrow K_1(C(L_q(d)) \longrightarrow \zz^{l+1} \,  \ar[r]^-{\, \, 1 - M^d \, \, } & \, \zz^{l+1} \longrightarrow K_0(C(L_q(d)) \longrightarrow 0 
}
\]
while the one in \eqref{eq:sixexamodII} becomes
\[
\xymatrix{ 
0 \longleftarrow K^1(C(L_q(d))  \longleftarrow  \zz^{l+1} \,  & \, \zz^{l+1} \ar[l]_-{\, \,  1 -( M^t)^d \, \, }  \longleftarrow K^0(C(L_q(d)) 
\longleftarrow 0 \, .
}
\]

Let us use the notation $\io : \zz \to \zz^l$, $1 \mapsto (1,\ldots,1)$ for the diagonal inclusion and let $\io^t : \zz^l \to \zz$ denote the transpose, $\io^t : (m_1,\ldots,m_l) \mapsto m_1 \plp m_l$.

\begin{theorem}
Let $k,l \in \nn$ be coprime and let $d \in \nn$. Then
\[
\begin{split}
& K_0\big( C(L_q(dlk;k,l)) \big) \simeq \T{Coker}(1 - M^d) \simeq \zz \op \big( \zz^l / \T{Im}(d \cd \io) \big) \\
& K_1\big(C(L_q(dlk;k,l))\big) \simeq \T{Ker}(1 - M^d) \simeq  \zz^l \\
\end{split}
\]
and
\[
\begin{split}
& K^0\big( C(L_q(dlk;k,l)) \big) \simeq \T{Ker}(1 - (M^t)^d ) \simeq \zz \op \big( \T{Ker}(\io^t) \big) \\
& K^1\big( C(L_q(dlk;k,l)) \big) \simeq \T{Coker}(1 - (M^t)^d) \simeq \zz/(d \zz) \op \zz^l  \, .
\end{split}
\]
\end{theorem} 

We finish by stressing that the results on the $K$-theory and $K$-homology of the lens spaces $L_q(dlk;k,l)$ are different from the ones obtained 
for instance in \cite{HS03}. In fact our lens spaces are not included in the class of lens spaces considered there.
Thus, for the moment, there seems to be no alternative method which results in a computation of the $KK$-groups of these spaces.


\begin{thebibliography}{9999}


\bibitem{AA94}
\textsc{A.~Al Amrani}, \emph{Complex K-theory of weighted projective spaces}, J. Pure Appl. Algebra 93 (1994), 113--127.

\bibitem{ABL14}
\textsc{F.~Arici, S.~Brain, G.~Landi}, \emph{The Gysin sequence for Quantum Lens Spaces}, arXiv:1401.6788; 
J. Noncomm. Geom. in press.

\bibitem{AEE:MEC}
\textsc{B.~Abadie, S.~Eilers, R.~Exel}, \emph{Morita
  equivalence for crossed products by {H}ilbert {$C\sp *$}-bimodules}, Trans.
  Amer. Math. Soc. 350 (1998), 3043--3054. 
  
\bibitem{Bl}
\textsc{B.~Blackadar}, \textit{$K$-Theory for Operator Algebras}, 2nd ed.,
Cambridge University Press, 1998.

\bibitem{BrFa:QT}
\textsc{T.~Brzezi{\'n}ski, S.A.~Fairfax}, \emph{Quantum
  teardrops}, Comm. Math. Phys. 316 (2012), 151--170.

\bibitem{BrMa:QGQ}
\textsc{T.~Brzezi{\'n}ski, S.~Majid}, \emph{Quantum group gauge
  theory on quantum spaces}, Comm. Math. Phys. 157 (1993), 591--638; Erratum 167 (1995), 235.

\bibitem{Co94}
\textsc{A.~Connes}, \textit{Noncommutative Geometry}, Academic Press, 1994.

\bibitem{Cun:SGI}
\textsc{J.~Cuntz}, \emph{Simple {$C\sp*$}-algebras generated by isometries},
  Comm. Math. Phys. 57 (1977), 173--185. 
  
\bibitem{CuKr:TMC}
\textsc{J.~Cuntz} and \textsc{W.~Krieger}, \emph{A class of {$C\sp{\ast}
  $}-algebras and topological {M}arkov chains}, Invent. Math. 56 (1980), 251--268.  

\bibitem{Exe:CPP}
\textsc{R.~Exel}, \emph{Circle actions on {$C\sp *$}-algebras, partial automorphisms, and a generalized {P}imsner-{V}oiculescu 
exact sequence}, J. Funct. Anal. 122 (1994), 361--401.  

\bibitem{GG13}
\textsc{O.~Gabriel, M.~Grensing}, \emph{Spectral triples and generalized crossed product}, arXiv:1310.5993.

\bibitem{Haj:SCQ}
\textsc{P.M. Hajac}, \emph{Strong connections on quantum principal bundles},
  Comm. Math. Phys. 182 (1996), 579--617. 
  
\bibitem{HS03}
\textsc{J.H.~Hong, W.~Szyma{\'n}ski}, \emph{Quantum Lens Spaces and Graph Algebras}, Pac. J. Math. 211 (2003), 249--263.


\bibitem{Ka78} 
\textsc{M.~Karoubi}, \emph{K-Theory: an Introduction}. Grund. der math. Wiss.  226, Springer, 1978.

\bibitem{Ka04} 
\textsc{T.~Katsura}, \emph{On {$C\sp *$}-algebras associated with {$C\sp *$}-correspondences}, J. Funct. Anal. 217 (2004), 366--401.

\bibitem{La95}
\textsc{E.C.~Lance}, {\em Hilbert C*-Modules: a Toolkit for Operator Algebraists}, London Mathematical Society Lecture Notes Series 210, Cambridge University Press, 1995.

\bibitem{NaVO82}
\textsc{C.~Nastasescu, F.~Van~Oystaeyen}, {\em Graded Ring Theory}, Elsevier, 1982.

\bibitem{Pim:CCC}
\textsc{M.V.~Pimsner}, \emph{A class of {$C\sp *$}-algebras generalizing both
  {C}untz-{K}rieger algebras and crossed products by {${\bf Z}$}}, Free
  probability theory ({W}aterloo, {ON}, 1995), Fields Inst. Commun., vol.~12,
  American Mathematical Society, 1997, 189--212. 
  
\bibitem{PiVo:EKE}
\textsc{M.~Pimsner} and \textsc{D.~Voiculescu}, \emph{Exact sequences for
  {$K$}-groups and {E}xt-groups of certain cross-product {$C\sp{\ast}
  $}-algebras}, J. Operator Theory 4 (1980),  93--118.  
 
\bibitem{Sim:TIA}
\textsc{B.~Simon}, \emph{Trace ideals and their applications}, 2nd ed.,
  Mathematical Surveys and Monographs, vol. 120, American Mathematical Society, 2005. 

\bibitem{SiVe:RDN} 
\textsc{A.~Sitarz, J.J.~Venselaar}, \emph{Real Spectral Triples on 3-dimensional Noncommutative Lens Spaces},
arXiv:1312.5690.

\bibitem{VS91}
\textsc{L.~Vaksman, Y.~Soibelman}, \emph{The Algebra of Functions on the Quantum
 Group $\textup{SU}(n+1)$ and Odd-Dimensional Quantum Spheres}, Leningrad Math. J. 2 (1991), 
 1023--1042.

\bibitem{Wor88}
\textsc{S.L.~Woronowicz}, 
\textit{Tannaka-Krein Duality for Compact Matrix Pseudogroups. Twisted $\textup{SU}(N)$ group}, Inv. Math. 93 (1988), 35--76. 

\end{thebibliography}
\end{document}